\documentclass[11pt,a4paper]{article}

\usepackage{titlesec}
\usepackage{fancyhdr}
\usepackage{a4wide}
\usepackage{graphicx}
\usepackage{float}
\usepackage{amssymb}
\usepackage{amsmath}
\usepackage{amsthm}
\usepackage{xcolor,colortbl}
\usepackage{mathrsfs}
\usepackage{array}
\usepackage{eucal}
\usepackage{tikz}
\usetikzlibrary{patterns}
\usepackage{multido}
\usepackage[T1]{fontenc}
\usepackage{hhline}
\usepackage{inputenc}
\usepackage[english]{babel}
\usetikzlibrary{shapes.misc}
\tikzset{cross/.style={cross out, draw=blue, minimum size=1*(#1-\pgflinewidth), inner sep=0pt, outer sep=0pt},
cross/.default={6pt}}
\usepackage{lmodern}
\usepackage{hyperref}
\usepackage{geometry}
\usepackage{changepage}
\changepage{0pt}{}{}{}{}{0pt}{}{0pt}{10pt}
\usepackage[numbers]{natbib}
\usepackage{hyperref}
\hypersetup{
pdfpagemode=none,
pdftoolbar=true,        % show Acrobat's toolbar?
pdfmenubar=true,        % show Acrobats menu?
pdffitwindow=false,     % window fit to page when opened
pdfstartview={Fit},    % fits the width of the page to the window
pdftitle={On quasi-reversibility solutions to the Cauchy problem for the Laplace equation: regularity and error estimates},    % title
pdfauthor={L. Bourgeois, L. Chesnel},     % author
pdfsubject={},  % subject of the document
pdfcreator={L. Bourgeois, L. Chesnel},   % creator of the document
pdfproducer={L. Bourgeois, L. Chesnel}, % producer of the document
pdfkeywords={}, % list of keywords
pdfnewwindow=true,      % links in new window
colorlinks=true,       % false: boxed links; true: colored links
linkcolor=magenta,          % color of internal links
citecolor=red,        % color of links to bibliography
filecolor=cyan,      % color of file links
urlcolor=blue           % color of external links
}

\newcommand{\dsp}{\displaystyle}

\newcommand{\eps}{\varepsilon}
\newcommand{\om}{\omega}
\newcommand{\Om}{\Omega}

\newcommand{\R}{\mathbb{R}}

\definecolor{Gray}{gray}{0.90}
\newcommand {\be}{\begin{equation}}
\newcommand {\ee}{\end{equation}}

\newcommand {\la} {\lambda}

\newtheorem{theorem}{Theorem}[section]
\newtheorem{lemma}{Lemma}[section]
\newtheorem{remark}{Remark}[section]

\newtheorem{corollary}{Corollary}[section]
\newtheorem{proposition}{Proposition}[section]

\begin{document}

~\vspace{-0.3cm}
\begin{center}
{\sc \bf\LARGE On quasi-reversibility solutions to the Cauchy problem \\[6pt] for the Laplace equation: regularity and error estimates}
\end{center}

\begin{center}
\textsc{Laurent Bourgeois}$^1$, \textsc{Lucas Chesnel}$^{2}$\\[16pt]
\begin{minipage}{0.95\textwidth}
{\small
$^1$ Laboratoire Poems, CNRS/ENSTA/INRIA, Ensta ParisTech, Universit\'e Paris-Saclay, 828, Boulevard des Mar\'echaux, 91762 Palaiseau, France; \\
$^2$ INRIA/Centre de math\'ematiques appliqu\'ees, \'Ecole Polytechnique, Universit\'e Paris-Saclay, Universit\'e Paris-Saclay, Route de Saclay, 91128 Palaiseau, France.\\[10pt]
E-mails: \textit{Laurent.Bourgeois@ensta-paristech.fr},\, \textit{Lucas.Chesnel@inria.fr}\\[-14pt]
\begin{center}
(\today)
\end{center}
}
\end{minipage}
\end{center}
\vspace{0.4cm}

\noindent\textbf{Abstract.} We are interested in the classical ill-posed Cauchy problem for the Laplace equation. One method to approximate the solution associated with compatible data consists in considering a family of regularized well-posed problems depending on a small parameter $\eps>0$. In this context, in order to prove convergence of finite elements methods, it is necessary to get regularity results of the solutions to these regularized problems which hold uniformly in $\eps$. In the present work, we obtain these results in smooth domains and in 2D polygonal geometries. In presence of corners, due the particular structure of the regularized problems, classical techniques \textit{\`a la} Grisvard do not work and instead, we apply the Kondratiev approach. We describe the procedure in detail to keep track of the dependence in $\eps$ in all the estimates. The main originality of this study lies in the fact that the limit problem is ill-posed in any framework. 

\section{Introduction and setting of the problem}
\label{introduction}

\begin{figure}[!ht]
\centering
\begin{tikzpicture}[scale=3]
\draw [draw=black,fill=gray!20] plot [smooth cycle, tension=1] coordinates {(-0.6,0.9) (0,0.5) (0.7,1) (0.5,1.5) (-0.2,1.4)};
\node at (0,0.9){ $\Om$};
\node at (-0.6,0.6){ $\Gamma$};
\node at (-0.3,1.45){ $\tilde{\Gamma}$};
\path [clip] (-1,0.4) rectangle (1,1);
\draw [draw=blue,line width=0.7mm] plot [smooth cycle, tension=1] coordinates {(-0.6,0.9) (0,0.5) (0.7,1) (0.5,1.5) (-0.2,1.4)};
\end{tikzpicture}\qquad\quad\begin{tikzpicture}[scale=3]
\draw [draw=blue,line width=0.7mm,fill=gray!20] plot [smooth cycle, tension=1] coordinates {(-0.7,1.1) (0,0.3) (0.7,1) (0.5,1.5) (-0.2,1.4)};
\node at (-0.4,0.9){ $\Om$};
\node at (-0.6,0.6){ $\Gamma$};
\begin{scope}[yshift=0.7cm]
\draw [draw=black,fill=white] plot [smooth cycle, tension=1] coordinates {(-0.1,0.5) (0.3,0.4) (0.1,0.1) (-0.2,0.1)};
\node at (-0.1,0.3){ $\tilde{\Gamma}$};
\end{scope}
\end{tikzpicture}
\caption{Examples of domains $\Om$. The thick blue lines represent the support of measurements.  \label{SchematicPicture}
} 
\label{figIntro}
\end{figure}
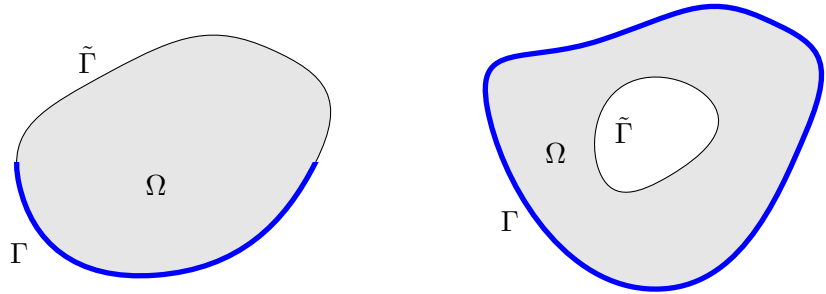

Let us consider a bounded Lipschitz domain $\Omega \subset \mathbb{R}^d$, $d>1$, the boundary $\partial \Omega$ of which is partitioned into two sets $\Gamma$ and $\tilde{\Gamma}$.
More precisely, $\Gamma$ and $\tilde{\Gamma}$ are non empty open sets for the topology induced on $\partial \Omega$ from the topology on $\mathbb{R}^d$, $\partial \Omega= \overline{\Gamma} \cup \overline{\tilde{\Gamma}}$ and $\Gamma \cap \tilde{\Gamma}=\emptyset$ (see Figure \ref{SchematicPicture}).
The Cauchy problem we are interested in consists, for some data $(g_0,g_1) \in H^{1/2}(\Gamma) \times H^{-1/2}(\Gamma)$, in finding $u \in H^1(\Omega)$ such that 
\be \left\{
\begin{array}{rccl}
 \Delta u  &=& 0&  \mbox{in }\Omega\\
 u&=&g_0  & \mbox{on }\Gamma \\
 \partial_\nu u&=&g_1 & \mbox{on }\Gamma,
\end{array}
\right.
\label{cauchy_strong}
\ee
where $\nu$ is the outward unit normal to $\partial\Omega$.
This kind of problem arises when some part $\tilde{\Gamma}$ of the boundary of a structure is not accessible, while the complementary part $\Gamma$ is the support of measurements which provide the Cauchy data $(g_0,g_1)$. It is important to note that in practice those measurements are contaminated by some noise. 
Due to Holmgren's theorem, the Cauchy problem (\ref{cauchy_strong}) has at most one solution. However it is ill-posed in the sense of Hadamard: existence may not hold for some data $(g_0,g_1)$, as for example shown in \cite{benbelgacem}.
A possibility to regularize problem (\ref{cauchy_strong}) is to use the quasi-reversibility method, which goes back to \cite{lattes_lions} and was revisited in \cite{klibanov_santosa}.
The original idea was to replace an ill-posed Boundary Value Problem such as (\ref{cauchy_strong}) by a family, depending on a small parameter $\eps$, 
of well-posed fourth-order BVPs.
Much later, the first author introduced the notion
of mixed formulation of quasi-reversibility for the Cauchy problem of the Laplace equation \cite{bourgeois}. This notion was extended to general abstract linear ill-posed problems in 
\cite{bourgeois_recoquillay}. The idea is to replace the ill-posed second-order BVP by a family, again depending on a small parameter $\eps$, of second-order systems of two coupled 
BVPs: the advantage is that the order of the regularized problem is the same as the original one, which is interesting when it comes to the numerical resolution. The price to pay is the introduction of a second unknown function $\la_\eps$ in addition to the principal unknown $u_\eps$.
Such mixed formulation of quasi-reversibility is the following: for $\eps >0$, find $(u_\eps,\lambda_\eps) \in V_{g_0} \times \tilde{V}_0$ such that for all $(v,\mu) \in V_0 \times \tilde{V}_0$,
\begin{equation} 
\left\{\begin{array}{rcl} 
\displaystyle \eps \int_\Omega \nabla u_\eps \cdot \nabla v\,dx + \int_\Omega \nabla v\cdot \nabla \lambda_\eps\,dx&=&0\\[8pt]
\displaystyle \int_\Omega \nabla u_\eps \cdot \nabla \mu\,dx - \int_\Omega \nabla \lambda_\eps\cdot \nabla \mu\,dx&=&\langle g_1,\mu \rangle_{H^{-1/2}(\Gamma),\tilde{H}^{1/2}(\Gamma)},
\end{array} \right. 
\label{Pr_cauchy} 
\end{equation}
where $V_{g_0}=\{u \in H^1(\Omega),\,\,u|_{\Gamma}=g_0\}$, $V_0=\{u \in H^1(\Omega),\,\,u|_{\Gamma}=0\}$ and
$\tilde{V}_0=\{\la \in H^1(\Omega),\,\,\la|_{\tilde{\Gamma}}=0\}$. In (\ref{Pr_cauchy}), the brackets stand for duality pairing between $H^{-1/2}(\Gamma)$ and $\tilde{H}^{1/2}(\Gamma)$. 
Here $\tilde{H}^{1/2}(\Gamma)$ is the subspace formed by the functions in $H^{1/2}(\Gamma)$ which, once extended by $0$ on $\partial \Omega$, remain in $H^{1/2}(\partial \Omega)$.
We observe that in view of Poincar\'e inequality, the standard norm of $H^1(\Omega)$ in the spaces $V_0$ and $\tilde{V}_0$ is equivalent to the semi-norm $\|\cdot\|$ defined by
$\|\cdot\|^2=\int_\Omega |\nabla \cdot|^2\,dx$.
Let us denote $(\cdot,\cdot)$ the corresponding scalar product.
We remark that the weak formulation (\ref{Pr_cauchy}) is equivalent to the strong problem
\be \left\{
\begin{array}{rcll}
 \Delta u_\eps  &= &0&  \mbox{in }\Omega\\
 \Delta \la_\eps  &= &0&  \mbox{in }\Omega\\
 u_\eps&=&g_0  & \mbox{on }\Gamma \\
 \partial_\nu u_\eps -\partial_\nu \la_\eps&=&g_1 & \mbox{on }\Gamma \\
 \la_\eps&=&0  & \mbox{on }\tilde{\Gamma} \\
 \eps\, \partial_\nu u_\eps + \partial_\nu \la_\eps&=&0 & \mbox{on }\tilde{\Gamma},
\end{array}
\right.
\label{Pr_cauchy_strong}
\ee
where we observe that the two unknowns $u_\eps$ and $\la_\eps$ are harmonic functions which are coupled by the boundary $\partial \Omega$.
We have the following theorem.
\begin{theorem}
\label{base}
For all $(g_0,g_1) \in H^{1/2}(\Gamma) \times H^{-1/2}(\Gamma)$, the problem (\ref{Pr_cauchy}) has a unique solution $(u_\eps,\la_\eps) \in V_{g_0} \times \tilde{V}_0$.
There exists a constant $C$ which depends only on the geometry such that
\[\forall \eps \in (0,1], \quad \sqrt{\eps}\|u_\eps\|_{H^1(\Omega)} + \|\la_\eps\|_{H^1(\Omega)} \leq C (\|g_0\|_{H^{1/2}(\Gamma)} +\|g_1\|_{H^{-1/2}(\Gamma)} ).\]
If in addition we assume that $(g_0,g_1)$ is such that problem (\ref{cauchy_strong}) has a (unique) solution $u$ (the data are said to be compatible), then there exists a constant $C$ which depends only on the geometry such that
\[\forall \eps >0,\quad \|u_\eps\|_{H^1(\Omega)} + \frac{\|\la_\eps\|_{H^1(\Omega)}}{\sqrt{\eps}} \leq C\|u\|_{H^1(\Omega)}\]
and
\[\lim_{\eps \rightarrow 0} \|u_\eps-u\|_{H^1(\Omega)}=0.\]
\end{theorem}
\noindent To prove such theorem, we need the following Lemma, which establishes an equivalent weak formulation to problem (\ref{cauchy_strong}) and which is proved in \cite{bourgeois_recoquillay}.
\begin{lemma}
For $(g_0,g_1) \in H^{1/2}(\Gamma) \times H^{-1/2}(\Gamma)$, the function $u \in H^1(\Omega)$ is a solution to problem (\ref{cauchy_strong}) if and only if
$u|_{\Gamma}=g_0$ and for all $\mu \in H^1(\Omega)$ with $\mu|_{\tilde{\Gamma}}=0$, we have 
\be \int_\Omega \nabla u \cdot \nabla \mu\ \,dx= \langle g_1,\mu \rangle_{H^{-1/2}(\Gamma),\tilde{H}^{1/2}(\Gamma)}. \label{cauchy_weak}\ee
\end{lemma}
\begin{proof}[Proof of Theorem \ref{base}]
Let us begin with the first part of the theorem.
There exists a continuous lifting operator  $g_0 \mapsto U$ from $H^{1/2}(\Gamma)$ to $H^1(\Omega)$ such that $U|_{\Gamma}=g_0$.
Let us define $\hat{u}_\eps=u_\eps-U \in V_0$.
By replacing in (\ref{Pr_cauchy}), we obtain that $(\hat{u}_\eps,\la_\eps) \in V_0 \times \tilde{V}_0$ satisfies, for all $(v,\mu) \in V_0 \times \tilde{V}_0$, the
system
\[
\left\{\begin{array}{rcl} 
\displaystyle \eps \int_\Omega \nabla \hat{u}_\eps \cdot \nabla v\,dx + \int_\Omega \nabla v\cdot \nabla \lambda_\eps\,dx&=&\dsp-\eps \int_\Omega \nabla U\cdot \nabla v\,dx\\[8pt]
\displaystyle \int_\Omega \nabla \hat{u}_\eps \cdot \nabla \mu\,dx - \int_\Omega \nabla \lambda_\eps\cdot \nabla \mu\,dx&=&\dsp\langle g_1,\mu \rangle_{H^{-1/2}(\Gamma),\tilde{H}^{1/2}(\Gamma)}- \int_{\Omega} \nabla U \cdot \nabla \mu\,dx.
\end{array} \right. 
\]
Well-posedness then relies on the Lax-Milgram Lemma applied to the coercive bilinear form
\[A_\eps((u,\la);(v,\mu))= \eps \int_\Omega \nabla u \cdot \nabla v\,dx + \int_\Omega \nabla v\cdot \nabla \lambda\,dx
-\int_\Omega \nabla u \cdot \nabla \mu\,dx + \int_\Omega \nabla \lambda\cdot \nabla \mu\,dx
\]
on $V_0 \times \tilde{V}_0$.
Choosing $v=\hat{u}_\eps$ and $\mu=\la_\eps$ and subtracting the two above equations, we obtain 
\[\eps \int_{\Omega} |\nabla \hat{u}_\eps|^2\,dx + \int_{\Omega} |\nabla \la_\eps|^2\,dx=-\eps \int_\Omega \nabla U \cdot \nabla \hat{u}_\eps\,dx-\langle g_1,\la_\eps \rangle+ \int_\Omega \nabla U\cdot \nabla \la_\eps\,dx.
\]
The Cauchy-Schwarz inequality implies
\[
\eps \|\hat{u}_\eps\|^2 +  \|\la_\eps\|^2 
\leq  \eps \| U\| \|\hat{u}_\eps\| + 
\|g_1\|_{H^{-1/2}(\Gamma)} \|\la_\eps\|_{H^{1/2}(\Gamma)} +  \| U\| \|\la_\eps\|.
\]
The equivalence of norm $\|\cdot\|$ and the standard $H^1(\Omega)$ norm in spaces $V_0$ and $\tilde{V}_0$, the continuity of the trace operator and the continuity of the lifting operator $g_0 \mapsto U$ yield
\[
\eps \|\hat{u}_\eps\|^2_{H^1(\Omega)} +  \|\la_\eps\|^2_{H^1(\Omega)}\leq  C\eps \|g_0\|_{H^{1/2}(\Gamma)} \|\hat{u}_\eps\|_{H^1(\Omega)} + 
(c\|g_1\|_{H^{-1/2}(\Gamma)} +  C\|g_0\|_{H^{1/2}(\Gamma)})\|\la_\eps\|_{H^1(\Omega)}.
\]
Using the Young's inequality to deal with the right hand side of the above inequality, the result follows.
Let us prove the second part of the theorem.
In the case when the Cauchy data $(g_0,g_1)$ is associated with the solution $u$, then $u$ satisfies the weak formulation (\ref{cauchy_weak}).
By subtracting (\ref{cauchy_weak}) to the second equation of (\ref{Pr_cauchy}), we obtain that for all $\mu \in \tilde{V}_0$,
\be \int_\Omega \nabla (u_\eps-u) \cdot \nabla \mu\,dx - \int_\Omega \nabla \lambda_\eps\cdot \nabla \mu\,dx=0.\label{second}\ee
Now setting $v=u_\eps-u \in V_0$ in the first equation of (\ref{Pr_cauchy}), setting $\mu=\la_\eps \in \tilde{V}_0$ in equation (\ref{second}) and subtracting the two obtained equations, we get
\[\eps \int_\Omega \nabla u_\eps \cdot \nabla (u_\eps-u)\,dx + \int_\Omega |\nabla \lambda_\eps|^2\,dx=0.\]
We deduce that the term $(u_\eps,u_\eps-u)$ in the above sum is nonpositive, which from the Cauchy-Schwarz inequality implies
that $\|u_\eps\| \leq \|u\|$ and then $\|\la_\eps\| \leq \sqrt{\eps} \|u\|$.
Hence there exists a constant $C$ such that 
\[\|u_\eps\|_{H^1(\Omega)} \leq C\|u\|_{H^1(\Omega)}\qquad\mbox{ and }\qquad \|\la_\eps\|_{H^1(\Omega)} \leq C\sqrt{\eps} \|u\|_{H^1(\Omega)}.\]
It remains to prove that $u_\eps \rightarrow u$ in $H^1(\Omega)$ when $\eps \rightarrow 0$.
The sequence $(u_\eps)$ is bounded in $H^1(\Omega)$. Therefore, there exists a subsequence, still denoted $(u_\eps)$, such that $u_\eps \rightharpoonup w$ in $H^1(\Omega)$ when $\eps \rightarrow 0$, with $w \in H^1(\Omega)$. Since the affine space $V_{g_0}$ is convex and closed, it is weakly closed. This guarantees that $w \in V_{g_0}$.
Besides, by passing to the limit in the second equation of (\ref{Pr_cauchy}) we obtain that $w$ satisfies the weak formulation (\ref{cauchy_weak}). Uniqueness in problem (\ref{cauchy_strong}) then implies that $w=u$, so that $(u_\eps)$ weakly converges to $u$ in $H^1(\Omega)$.
But
\[\|u_\eps-u\|^2=(u_\eps,u_\eps-u)-(u,u_\eps-u) \leq -(u,u_\eps-u),\]
so that weak convergence implies strong convergence.
Lastly, a standard contradiction argument enables us to conclude that all the sequence $(u_\eps)$ strongly converges to $u$ in $H^1(\Omega)$.
\end{proof}
\begin{remark}
Let us mention that another type of mixed formulation of quasi-reversibility was introduced in \cite{darde_finlandais}, 
in which the additional unknown lies in $H_{\rm div}(\Omega)$ instead of $H^1(\Omega)$. In addition, a notion of iterative formulation of quasi-reversibility was introduced and analyzed in \cite{darde}.
We believe that the quasi-reversibility formulation (\ref{Pr_cauchy}) is the easiest one to handle to establish regularity results of the weak solutions.
\end{remark} 
\noindent The estimates of Theorem \ref{base} involve $H^1(\Omega)$ norms of the regularized solution $(u_\eps,\la_\eps)$ in the case of a Lipschitz domain $\Omega$ and for the natural regularity of the Cauchy data $(g_0,g_1)$, that is $H^{1/2}(\Gamma) \times H^{-1/2}(\Gamma)$.
These estimates were derived in two different cases: the data $(g_0,g_1)$ are compatible or not.
The main concern of this paper is to analyze, when the domain $\Omega$ and the Cauchy data $(g_0,g_1)$ are more regular than Lipschitz
and $H^{1/2}(\Gamma) \times H^{-1/2}(\Gamma)$, respectively, the additional regularity of the solution $(u_\eps,\la_\eps)$, whether the data $(g_0,g_1)$ are compatible or not. 
We also want to obtain estimates in the corresponding norms.
In order to simplify the analysis, the additional regularity of the data $(g_0,g_1)$ is formulated in the following way: we assume that $(g_0,g_1)$ is such that there exists a function $U$ in $H^2(\Omega)$ with $(U|_{\Gamma},\partial_\nu U|_{\Gamma})=(g_0,g_1)$ and that we can define a continuous lifting operator $(g_0,g_1) \mapsto U$. 
Denoting $f=\Delta U \in L^2(\Omega)$ and considering the new translated unknown $u-U \rightarrow u$, the initial Cauchy problem (\ref{cauchy_strong}) can be transformed
into a homogeneous one (however still ill-posed): for $f \in L^2(\Omega)$, 
find $u \in H^1(\Omega)$ such that 
\be \left\{
\begin{array}{cccc}
 -\Delta u  &=& f&  \mbox{in } \Omega\\
 u&=&0  & \mbox{on } \Gamma \\
 \partial_\nu u&=&0 & \mbox{on } \Gamma.
\end{array}
\right.
\label{cauchy_strong_simple}
\ee
We emphasize that this regularity assumption made on the data is not an assumption of regularity of the solution $u$. It is simple to construct smooth data in the sense above such that the corresponding $u$ is only in $H^1(\Om)$ and not in $H^2(\Om)$. The mixed formulation of quasi-reversibility for problem (\ref{cauchy_strong_simple}) takes the following form: for $\eps >0$, find $(u_\eps,\lambda_\eps) \in V_0 \times \tilde{V}_0$ such that for all $(v,\mu) \in V_0 \times \tilde{V}_0$,
\begin{equation} 
\left\{\begin{array}{rcl} 
\displaystyle \eps \int_\Omega \nabla u_\eps \cdot \nabla v\,dx + \int_\Omega \nabla v\cdot \nabla \lambda_\eps\,dx&=&0\\[8pt]
\displaystyle \int_\Omega \nabla u_\eps \cdot \nabla \mu\,dx - \int_\Omega \nabla \lambda_\eps\cdot \nabla \mu\,dx&=&\dsp\int_\Omega f\mu\,dx.
\end{array} \right. 
\label{Pr_cauchy_simple} 
\end{equation}
Note that the strong equations corresponding to problem (\ref{Pr_cauchy_simple}) are
\be \left\{
\begin{array}{rcll}
\displaystyle  -(1+\eps)\Delta u_\eps  &=& f&  \mbox{in }\Omega\\
\displaystyle -(1+\eps)\Delta \la_\eps  &=& -\eps\,f&  \mbox{in }\Omega\\
 u_\eps&=&0  & \mbox{on }\Gamma \\
 \partial_\nu u_\eps -\partial_\nu \la_\eps&=&0 & \mbox{on }\Gamma \\
 \la_\eps&=&0  &  \mbox{on }\tilde{\Gamma} \\
 \eps\, \partial_\nu u_\eps + \partial_\nu \la_\eps&=&0 & \mbox{on }\tilde{\Gamma}.
\end{array}
\right.
\label{Pr_cauchy_simple_strong}
\ee
The analog of Theorem \ref{base}, the proof of which is skipped, is the following. 
\begin{theorem}
\label{base_simple}
For all $f \in L^2(\Omega)$ and $\eps>0$, the problem (\ref{Pr_cauchy_simple}) has a unique solution $(u_\eps,\la_\eps) \in V_{0} \times \tilde{V}_0$.
There exists a constant $C$ which depends only on the geometry such that
\be \forall \eps \in (0,1], \quad \sqrt{\eps}\|u_\eps\|_{H^1(\Omega)} + \|\la_\eps\|_{H^1(\Omega)} \leq C \|f\|_{L^2(\Omega)}. \label{estim1}\ee
If in addition we assume that $f$ is such that problem (\ref{cauchy_strong_simple}) has a (unique) solution $u$, then there exists a constant $C$ which depends only on the geometry such that
\be \forall \eps>0,\quad \|u_\eps\|_{H^1(\Omega)} + \frac{\|\la_\eps\|_{H^1(\Omega)}}{\sqrt{\eps}}\leq C\|u\|_{H^1(\Omega)} \label{estim2}\ee
and
\[\lim_{\eps \rightarrow 0} \|u_\eps-u\|_{H^1(\Omega)}=0.\]
\end{theorem}
\noindent The objective is now to study the regularity of the solution $(u_\eps,\la_\eps)$ to problem (\ref{Pr_cauchy_simple}) and to complete the statements (\ref{estim1}) and (\ref{estim2}) of Theorem \ref{base_simple} by giving estimates in stronger norms. 
One objective, as will be seen in section \ref{application}, is the following.
In practice, one has to solve problem (\ref{Pr_cauchy_simple}) in the presence of two approximations. Firstly, the data $f$ is altered by some noise of amplitude $\delta$. Secondly, the problem (\ref{Pr_cauchy_simple}) is discretized, for instance with the help of a Finite Element Method based on a mesh of size $h$. It is then desirable to estimate the error between the approximated solution and the exact solution as a function of $\eps$, $\delta$ and $h$. Such error estimate for the $H^1(\Omega)$ norm needs the solution to be in a Sobolev space $H^s(\Omega)$, with $s>1$. 
It could be noted that in a recent contribution \cite{burman_larson_oksanen} (see also \cite{Burm13,Burm14,Burm17,BuHL18}), a discretized method was proposed in order to regularize the Cauchy
problem (\ref{cauchy_strong}) in the presence of noisy data without introducing a regularized problem such as (\ref{Pr_cauchy_simple}) at the continuous level. In some sense, the method of \cite{burman_larson_oksanen} relies on a single asymptotic parameter, that is $h$, instead of two in our method, that is $\eps$ and $h$.
However, we believe that from the theoretical point of view, the regularity of quasi-reversibility solutions is an interesting problem in itself.
To our best knowledge, it has never been investigated up to now.
The difficulty stems from the fact that we analyze the regularity of a problem involving a small parameter $\eps$ which degenerates when $\eps$ tends to $0$.
There are other contributions (see e.g. \cite{Kirs85,costabel_dauge,ChCN14,ChCN18,NaPo18,NaPT18}) where regularity results or asymptotic expansions are obtained in situations where the limit problem has a different nature from the regularized one. For example in \cite{costabel_dauge}, the authors study a mixed Neumann-Robin problem where the small parameter $\eps$ is the inverse of the Robin coefficient. But while both the perturbed problem and the limit one are well-posed in \cite{costabel_dauge}, only the perturbed problem is well-posed in our case, the limit problem being ill-posed (in any framework). Our contribution is original in this sense. In the present work, we study the regularity of the solution of the regularized problem as $\eps$ tends to zero. We emphasize that computing an asymptotic expansion of the solution with respect to $\eps$ and proving error estimates (for example as in \cite{Ilin92,MaNP00}) remains an open problem, the reason being that, due to the ill-posedness of the limit problem, no result of stability can be easily established. \\  
\newline
Our paper is organized as follows. First we consider the simple case of a smooth domain in Section \ref{SectionSmoothDomain}, where classical regularity results (see for example \cite{brezis}) can be used. The case of the polygonal domain is introduced
in Section \ref{SectionPolygonalDomain}, where we also analyze the regularity of the quasi-reversibility solution in corners delimited by two edges of $\Gamma$ or two edges of $\tilde{\Gamma}$. In this case, the regularity of functions $u_\eps$ and $\la_\eps$ can be analyzed separately with the help of the classical regularity results of \cite{grisvard_bleu} in a polygon for the Laplace equation with Dirichlet or Neumann boundary conditions. In Section \ref{CornerGammaGammaTilde} we consider the more difficult case of a corner of mixed type, that is delimited by one edge of $\Gamma$ and one edge of $\tilde{\Gamma}$. This analysis relies on the Kondratiev approach \cite{Kond67}, which is based on some properties of weighted Sobolev spaces which are recalled in Section \ref{SectionWeightedSobo}. Section \ref{application} is dedicated to the application of our regularity results to derive some error estimate between the exact solution and the quasi-reversibility solution in the presence of two perturbations: noisy data and discretization with the help of a Finite Element Method. Two appendices containing technical results, which are used in Section \ref{CornerGammaGammaTilde}, complete the paper. The main results of this article are Theorem \ref{global_smooth} (uniform regularity estimates in smooth domains), Theorem \ref{main} (uniform regularity estimates in 2D polygonal domains) and the final approximation analysis of Section \ref{application}.

\section{The case of a smooth domain}\label{SectionSmoothDomain}
Let us first assume that $\Omega$ is a domain of class $C^{1,1}$.
If $(g_0,g_1) \in H^{3/2}(\Gamma) \times H^{1/2}(\Gamma)$, then there exists a function $U \in H^2(\Omega)$ such that
$(U|_{\Gamma},\partial_\nu U|_\Gamma)=(g_0,g_1)$ and even a continuous lifting operator $(g_0,g_1) \mapsto U$ from $H^{3/2}(\Gamma) \times H^{1/2}(\Gamma)$ to $H^2(\Omega)$ (see Theorem 1.5.1.2 in \cite{Gris85}). We are therefore in the situation described in the Section \ref{introduction}, where the problem to solve is (\ref{cauchy_strong_simple}).
We begin with an interior regularity result.
\begin{proposition}
\label{interior}
For $f \in L^2(\Omega)$, the solution $(u_\eps,\la_\eps) \in V_{0} \times \tilde{V}_0$ to the problem (\ref{Pr_cauchy_simple}) is such that for all $\zeta\in \mathscr{C}_0^\infty(\Omega)$, $\zeta u_\eps$ and $\zeta \la_\eps$ belong to $H^2(\Omega)$ and there exists a constant $C>0$ which depends only on the geometry such that
\[\forall \eps \in (0,1], \quad \sqrt{\eps}\|\zeta u_\eps\|_{H^2(\Omega)} + \|\zeta \la_\eps\|_{H^2(\Omega)} \leq C \|f\|_{L^2(\Omega)}.\]
If in addition $f$ is such that problem (\ref{cauchy_strong_simple}) has a solution $u$, then
\[\forall \eps \in (0,1], \quad \|\zeta u_\eps\|_{H^2(\Omega)} + \frac{\|\zeta \la_\eps\|_{H^2(\Omega)}}{\sqrt{\eps}} \leq C \|u\|_{H^1(\Delta,\Omega)},\]
where  the norm $\|\cdot\|_{H^1(\Delta,\Omega)}$ is defined by
\[\|u\|^2_{H^1(\Delta,\Omega)}=\|u\|_{H^1(\Omega)}^2 + \|\Delta u\|_{L^2(\Omega)}^2.\]
\end{proposition}
\begin{proof}
From the first equation of (\ref{Pr_cauchy_simple_strong}), we have that
\[-\Delta (\zeta u_\eps)+ \zeta u_\eps=(-\Delta \zeta +\zeta)u_\eps -2\nabla \zeta \cdot \nabla u_\eps + \zeta \frac{f}{1+\eps}:=F_\eps.\]
Clearly $F_\eps \in L^2(\mathbb{R}^2)$, which by using the Fourier transform implies that
\[\|\zeta u_\eps\|_{H^2(\mathbb{R}^2)} = \|F_\eps\|_{L^2(\mathbb{R}^2)},\]
and hence
\[\|\zeta u_\eps\|_{H^2(\Omega)} = \|F_\eps\|_{L^2(\Omega)} \leq C\, (\|u_\eps\|_{H^1(\Omega)}+ \|f\|_{L^2(\Omega)}).\]
From (\ref{estim1}) we obtain that 
\[\sqrt{\eps}\|\zeta u_\eps\|_{H^2(\Omega)} \leq C\, \|f\|_{L^2(\Omega)}.\]
If in addition $f$ is such that problem (\ref{cauchy_strong_simple}) has a (unique) solution $u$, from (\ref{estim2}) we obtain
\[\|\zeta u_\eps\|_{H^2(\Omega)} \leq C\, \|u\|_{H^1(\Delta,\Omega)}.\]
The estimates of $\zeta \la_\eps$ are obtained following the same lines. 
\end{proof}
\noindent Let us now establish a global regularity estimate (up to the boundary) in the restricted case when $\overline{\Gamma} \cap \overline{\tilde{\Gamma}}= \emptyset$ (see Figure \ref{SchematicPicture} right).
\begin{theorem}\label{global_smooth}
For $f \in L^2(\Omega)$, the solution $(u_\eps,\la_\eps) \in V_{0} \times \tilde{V}_0$ to the problem (\ref{Pr_cauchy_simple}) is such that $u_\eps$ and $\la_\eps$ belong to $H^2(\Omega)$ and there exists a constant $C>0$ which depends only on the geometry such that
\[\forall \eps \in (0,1], \quad \eps\|u_\eps\|_{H^2(\Omega)} + \sqrt{\eps}\|\la_\eps\|_{H^2(\Omega)} \leq C \|f\|_{L^2(\Omega)}.\]
If in addition $f$ is such that problem (\ref{cauchy_strong_simple}) has a solution $u$, then
\[\forall \eps \in (0,1], \quad \sqrt{\eps}\|u_\eps\|_{H^2(\Omega)} + \|\la_\eps\|_{H^2(\Omega)} \leq C \|u\|_{H^1(\Delta,\Omega)}.\]
\end{theorem}
\begin{proof}
Given $\overline{\Gamma} \cap \overline{\tilde{\Gamma}}= \emptyset$, we may 
find two infinitely smooth functions $\zeta$ and $\tilde{\zeta}$ such that $(\zeta,\tilde{\zeta})=(1,0)$ in a vicinity of $\Gamma$ and
$(\zeta,\tilde{\zeta})=(0,1)$ in a vicinity of $\tilde{\Gamma}$.
We have from the first equation of (\ref{Pr_cauchy_simple_strong}),
\[-\Delta(\zeta u_\eps)=-\Delta \zeta u_\eps -2 \nabla \zeta\cdot \nabla u_\eps + \zeta \frac{f}{1+\eps}=F_\eps.\]
Since $u_\eps=0$ on $\Gamma$, from a standard regularity result for the Poisson equation with Dirichlet boundary condition we obtain 
\be \|\zeta u_\eps\|_{H^2(\Omega)} \leq C \|F_\eps\|_{L^2(\Omega)} \leq C(\|f\|_{L^2(\Omega)} + \|u_\eps\|_{H^1(\Omega)}),\label{F}\ee
and from (\ref{estim1}) we have
\[\sqrt{\eps} \|\zeta u_\eps\|_{H^2(\Omega)} \leq C\|f\|_{L^2(\Omega)}.\]
From a standard continuity result for the normal derivative and using $\partial_\nu u_\eps-\partial_\nu \la_\eps=0$ on $\Gamma$, we obtain 
\[\sqrt{\eps} \|\partial_\nu \la_\eps\|_{H^{1/2}(\Gamma)} =\sqrt{\eps} \|\partial_\nu u_\eps\|_{H^{1/2}(\Gamma)} \leq C \|f\|_{L^2(\Omega)}.\]
From the second equation of (\ref{Pr_cauchy_simple_strong}) we have
\[\|\Delta \la_\eps\|_{L^2(\Omega)} \leq C\eps\|f\|_{L^2(\Omega)}.\] 
Combining the two previous estimates with the fact that $\la_\eps=0$ on $\tilde{\Gamma}$ implies the regularity estimate 
\[\sqrt{\eps} \|\la_\eps\|_{H^2(\Omega)} \leq C \|f\|_{L^2(\Omega)}.\]
Reusing the second equation of (\ref{Pr_cauchy_simple_strong}), the estimate (\ref{estim1}) and that $\la_\eps=0$ on $\tilde{\Gamma}$ leads to
\[\|\tilde{\zeta} \la_\eps\|_{H^2(\Omega)} \leq C \|f\|_{L^2(\Omega)},\]
and using $\eps \partial_\nu u_\eps+\partial_\eps \la_\eps=0$ on $\tilde{\Gamma}$, we obtain 
\[\eps\|\partial_\nu u_\eps\|_{H^{1/2}(\tilde{\Gamma})} =\|\partial_\nu \la_\eps\|_{H^{1/2}(\tilde{\Gamma})} \leq C \|f\|_{L^2(\Omega)}.\]
We conclude that
\[\eps\|u_\eps\|_{H^2(\Omega)} \leq C \|f\|_{L^2(\Omega)}.\]
Now let us assume that $f$ is such that problem (\ref{cauchy_strong_simple}) has a solution $u$.
From (\ref{estim2}) and (\ref{F}) we now have
the better estimate
\[\|\zeta u_\eps\|_{H^2(\Omega)} \leq C\|u\|_{H^1(\Omega,\Delta)}.\]
Using $\partial_\nu u_\eps-\partial_\nu \la_\eps=0$ on $\Gamma$, we obtain 
\[\|\partial_\nu \la_\eps\|_{H^{1/2}(\Gamma)} \leq C \|u\|_{H^1(\Delta,\Omega)},\]
and then
\[\|\la_\eps\|_{H^2(\Omega)} \leq C \|u\|_{H^1(\Delta,\Omega)}.\]
Reusing the second equation of (\ref{Pr_cauchy_simple_strong}), the estimate (\ref{estim2}) and that $\la_\eps=0$ on $\tilde{\Gamma}$ leads to
\[\|\tilde{\zeta} \la_\eps\|_{H^2(\Omega)} \leq C \sqrt{\eps} \|u\|_{H^1(\Delta,\Omega)}.\]
Since $\eps \partial_\nu u_\eps+\partial_\eps \la_\eps=0$ on $\tilde{\Gamma}$, we obtain 
\[\sqrt{\eps}\|\partial_\nu u_\eps\|_{H^{1/2}(\tilde{\Gamma})} \leq C \|u\|_{H^1(\Delta,\Omega)}.\]
We conclude that $\sqrt{\eps}\|u_\eps\|_{H^2(\Omega)} \leq C \|u\|_{H^1(\Delta,\Omega)}$.
\end{proof}
\begin{remark}
From Theorem \ref{base} and Proposition \ref{interior}, we notice that in the interior of the domain, the $H^2$ estimates are the same as the $H^1$ estimates, whether the data are compatible or not.
However, from Theorem \ref{base} and Theorem \ref{global_smooth}, when it comes to the $H^2$ estimates in the whole domain, up to the boundary, one loses a $\sqrt{\eps}$ factor with respect to the $H^1$ estimates, whether the data are compatible or not.
\end{remark}
\section{The case of a polygonal domain}\label{SectionPolygonalDomain}
\subsection{Main result}

\begin{figure}[!ht]
\centering
\begin{tikzpicture}[scale=3]
\draw [draw=black,fill=gray!20] (-0.6,0.9)--(0,0.5)--(0.7,1)--(0.5,1.5)--(-0.2,1.4)--cycle;
\draw [->,domain=-35:46] plot ({-0.6+0.15*cos(\x)}, {0.9+0.15*sin(\x)});
\draw [draw=blue,line width=0.7mm] (-0.2,1.4)--(-0.6,0.9)--(0,0.5);
\node at (-0.35,0.9){ \small $\om_1$};
\node at (0.2,1.1){ $\Om$};
\node at (-0.72,0.9){ $S_1$};
\node at (0,0.4){ $S_2$};
\node at (0.8,0.9){ $S_3$};
\node at (-0.35,0.6){ $\Gamma$};
\node at (0.1,1.52){ $\tilde{\Gamma}$};

\end{tikzpicture}
\caption{An example of polygonal domain. $S_1$, $S_2$, $S_3$ represent the three types of vertices that we will study in \S\ref{CornerGamma}, \S\ref{CornerGammaGammaTilde}, \S\ref{CornerGammaTilde} respectively. \label{Polygone}
} 
\end{figure}
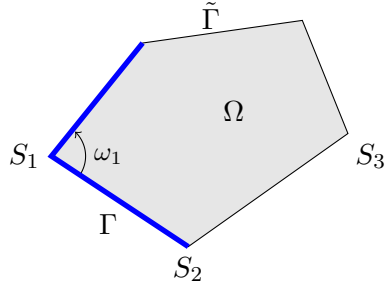

\noindent From now on, $\Omega$ is a polygonal domain in dimension 2. Our motivation is indeed to obtain error estimates in the context of the discretization with the help of a classical Finite Element Method: due to the meshing procedure in two dimensions, in practice the computational domain is often a polygon.
We use the same notations as in \cite{grisvard_bleu} to describe the geometry of such a polygon.
Let us assume that $\partial \Omega$ is the union of segments $\overline{\Gamma_j}$, $j=1,\dots,N$, where $N$ is an integer. Let us denote $S_j$ the vertex such that $S_j=\overline{\Gamma_j} \cap \overline{\Gamma_{j+1}}$, $\omega_j$ the angle between $\Gamma_j$
and $\Gamma_{j+1}$ from the interior of $\Omega$, $\tau_j$ the unit tangent oriented in the counter-clockwise sense and $\nu_j$ the outward normal to $\partial\Omega$.
We assume that $\Gamma$ and $\tilde{\Gamma}$ are formed by a finite number of edges, namely $n$ and $\tilde{n}$, respectively, with $n+\tilde{n}=N$.
Let us denote $\mathcal{H}(\Gamma)$ the subset of functions $(g_0,g_1) \in L^2(\Gamma) \times L^2(\Gamma)$ such that 
$(f_j,g_j):=(g_0|_{\Gamma_j},g_1|_{\Gamma_j}) \in H^{3/2}(\Gamma_j) \times H^{1/2}(\Gamma_j)$, $j=1,\dots,n$,
with the following compatibility conditions at $S_j$:
\be
\left\{
\begin{array}{lll}
f_j(S_j) &= & f_{j+1}(S_j)\\[4pt]
\partial_{\tau_j}f_j &\equiv& -\cos(\omega_j) \partial_{\tau_{j+1}}f_{j+1} + \sin(\omega_j) g_{j+1} \quad \text{at} \quad S_j \\[4pt]
 g_j &\equiv& -\sin(\omega_j) \partial_{\tau_{j+1}}f_{j+1} -\cos(\omega_j) g_{j+1} \quad \text{at} \quad S_j,
\end{array}
\right.
\ee
and the equivalence $\phi_j\equiv \phi_{j+1}$ at $S_j$ means that for small $\delta>0$
\[\int_0^{\delta} \frac{|\phi_j(x_j(-\sigma))-\phi_{j+1}(x_j(+\sigma))|^2}{\sigma}\,d\sigma <+\infty,\]
where $x_j(\sigma)$ denotes the point of $\partial \Omega$ which, for small enough $|\sigma|$ (say $|\sigma| \leq \delta$), is at distance $\sigma$ (counted algebraically) of $S_j$ along $\partial \Omega$. More precisely, $x_j(\sigma) \in \Gamma_j$ if $\sigma<0$ and $x_j(\sigma) \in \Gamma_{j+1}$
if $\sigma>0$.
It is proved in \cite{grisvard_bleu}, that for $(g_0,g_1) \in \mathcal{H}$, there exists a function $U \in H^2(\Omega)$ such that for each $j=1,\dots,n$, $(U|_{\Gamma_j},\partial_{\nu_j} U|_{\Gamma_j})=(f_j,g_j)$ and even a continuous lifting $(g_0,g_1) \mapsto U$ from $\mathcal{H}$ to $H^2(\Omega)$. We are hence again in the framework of section \ref{introduction}, where the problem to solve is (\ref{cauchy_strong_simple}).\\
\newline
Clearly, the interior estimates given by Proposition \ref{interior} are true in the polygonal domain since they are independent of the regularity of the boundary.
Let us now analyze the regularity up to the boundary.
As done in \cite{grisvard_bleu}, the estimates are obtained by using a partition of unity, which enables us to localize our analysis in three different types of corners (see Figure \ref{Polygone}): 
\begin{itemize}
\item regularity at a corner delimited by two edges which belong to $\Gamma$, called a corner of type $\Gamma$,
\item regularity at a corner delimited by two edges which belong to $\tilde{\Gamma}$, called a corner of type $\tilde{\Gamma}$,
\item regularity at a corner delimited by one edge which belongs to $\Gamma$ and one edge which belongs to $\tilde{\Gamma}$, called a corner of mixed type.
\end{itemize}
Let us denote by $N_C$ the set of $j$ such that $S_j$ is either a vertex of type $\Gamma$ or a vertex of type $\tilde{\Gamma}$ and $N_M$ the set of $j$ such that $S_j$ is a corner of mixed type.
We wish to prove the following theorem, which is obtained by gathering Propositions \ref{interior}, \ref{Gamma}, \ref{Gamma_tilde} and \ref{mixed} hereafter.
\begin{theorem}
\label{main}
Let us take $s_C<\min_{j \in N_C}(1+\pi/\omega_j)$ if there exists $j \in N_C$ such that $\omega_j>\pi$ and $s_C=2$ otherwise.
Let us take $s_M<\min_{j \in N_M}(1+\pi/(2\omega_j))$ if there exists $j \in N_M$ such that $\omega_j \geq \pi/2$ and $s_M=2$ otherwise.
Let us denote $s=\min(s_C,s_M)$.\\
\newline
For $f \in L^2(\Omega)$ and $\eps>0$, the solution $(u_\eps,\la_\eps) \in V_{0} \times \tilde{V}_0$ to the problem (\ref{Pr_cauchy_simple}) is such that $u_\eps$ and $\la_\eps$ belong to $H^{s}(\Omega)$ and there exists a constant $C>0$ which depends only on the geometry such that
\[\forall \eps \in (0,1], \quad \eps\|u_\eps\|_{H^{s}(\Omega)} + \sqrt{\eps}\|\la_\eps\|_{H^{s}(\Omega)} \leq C \|f\|_{L^2(\Omega)}.\]
If in addition we assume that $f$ is such that problem (\ref{cauchy_strong_simple}) has a (unique) solution $u$, then
\[\forall \eps \in (0,1], \quad \sqrt{\eps}\|u_\eps\|_{H^{s}(\Omega)} + \|\la_\eps\|_{H^{s}(\Omega)} \leq C \|u\|_{H^1(\Delta,\Omega)}.\]
\end{theorem}
\begin{remark}
The global estimates of Theorem \ref{main} are obtained by gathering all the local estimates obtained in Propositions \ref{interior}, \ref{Gamma}, \ref{Gamma_tilde} and \ref{mixed}. Each of these estimates are locally better than the global estimate of Theorem \ref{main}.
\end{remark}
\subsection{Regularity at a corner of type $\Gamma$}\label{CornerGamma}
The regularity of solutions $u_\eps$ and $\la_\eps$ near a corner delimited by two edges which belong to $\Gamma$ can be analyzed separately.
They will be obtained by directly applying the results of \cite{grisvard_bleu} for Dirichlet and Neumann Laplacian problems.
Let us consider $S_j$ the vertex of a corner delimited by two edges $\Gamma_j$ and $\Gamma_{j+1}$ which belong to $\Gamma$. 
Let us denote $(r_j,\theta_j)$ the local polar coordinates with respect to the point $S_j$ and $\zeta_j \in \mathscr{C}^\infty(\overline{\Omega})$ a radial function (depending only on $r_j$) such that $\zeta_j=1$ for $r_j \leq a_j$ and $\zeta_j=0$ for $r_j \geq b_j$. We assume that $b_j$ is chosen such that $\zeta_j=0$ in a vicinity of all edges $\Gamma_k$ except for $k=j$ or $k=j+1$. 
In order to simplify notations, we skip the reference to index $j$, denoting in particular $S_j=S$, $\Gamma_j=\Gamma_0$ and $\Gamma_{j+1}=\Gamma_\omega$.
Let us introduce the finite cone $K_b=\Omega \cap B(S,b)$.
The two following lemmata are proved in \cite{grisvard_bleu}.
\begin{lemma}
\label{Dirichlet}
For $F \in L^2(K_b)$, the problem: find $U \in H^1(K_b)$ such that
\be \left\{
\begin{array}{rcll}
 -\Delta U  &=& F&  \mbox{in }K_b\\
 U&=&0  & \mbox{on }\partial K_b
\end{array}
\right.
\label{dirichlet}
\ee
has a unique solution and there exists a unique constant $c \in \mathbb{R}$ and a unique function $V \in H^2(K_b)$
such that
\[U= c\, r^{\pi/\omega} \sin\left(\frac{\pi \theta}{\omega}\right) + V.\]
Moreover, there exists a constant $C>0$ such that
\[|c| + \|V\|_{H^2(K_b)} \leq C\, \|F\|_{L^2(K_b)} \]
In addition, if $\omega\le\pi$ then $c=0$.
\end{lemma} 
\begin{lemma}
\label{Neumann}
For $F \in L^2(K_b)$, the problem: find $U \in H^1(K_b)$ such that
\be \left\{
\begin{array}{rcll}
 -\Delta U  &=& F&  \mbox{in }K_b\\
U&=&0 & \mbox{on }\partial B(S,b) \cap \partial K_b \\
 \partial_{\nu} U&=&0  & \mbox{on } (\Gamma_0 \cup \Gamma_{\omega}) \cap \partial K_b
\end{array}
\right.
\label{neumann}
\ee
has a unique solution and there exists a unique constant $c \in \mathbb{R}$ and a unique function $V \in H^2(K_b)$
such that
\[ U =c\,r^{\pi/\omega} \cos\left(\frac{\pi \theta}{\omega}\right) + V.\]
Moreover, there exists a constant $C>0$ such that
\[|c| + \|V\|_{H^2(K_b)} \leq C\, \|F\|_{L^2(K_b)} \]
In addition, if $\omega\le\pi$ then $c=0$.
\end{lemma} 
\begin{proposition}
\label{Gamma}
Assume that $S$ is the vertex of a corner of type $\Gamma$.
Let us consider $s <1+\pi/\omega$ if $\omega>\pi$ and $s=2$ otherwise. 
For $f \in L^2(\Omega)$, the solution $(u_\eps,\la_\eps) \in V_{0} \times \tilde{V}_0$ to the problem (\ref{Pr_cauchy_simple}) is such that $\zeta u_\eps$ and $\zeta \la_\eps$ belong to $H^{s}(\Omega)$ and there exists a constant $C>0$ which depends only on the geometry such that
\[\forall \eps \in (0,1], \quad \sqrt{\eps}(\|\zeta u_\eps\|_{H^{s}(\Omega)} + \|\zeta \la_\eps\|_{H^{s}(\Omega)}) \leq C \|f\|_{L^2(\Omega)}.\]
If in addition $f$ is such that problem (\ref{cauchy_strong_simple}) has a solution $u$, then
\[\forall \eps \in (0,1], \quad \|\zeta u_\eps\|_{H^{s}(\Omega)} + \|\zeta \la_\eps\|_{H^{s}(\Omega)} \leq C \|u\|_{H^1(\Delta,\Omega)}.\]
\end{proposition}
\begin{proof}
From (\ref{Pr_cauchy_simple_strong}) we have that $\zeta u_\eps$ satisfies
problem (\ref{dirichlet}) with
\begin{equation}\label{DefSourceTerm}
F_{\eps}=-\Delta \zeta u_\eps-2\nabla \zeta\cdot \nabla u_\eps + \zeta \frac{f}{1+\eps}.
\end{equation}
By using Lemma \ref{Dirichlet}, we have that
there exists a unique constant $c_\eps\in \mathbb{R}$ and a unique function $V_\eps\in H^2(K_b)$
such that
\[\zeta(r) u_\eps=c_{\eps}\,r^{\pi/\omega} \sin\left(\frac{\pi \theta}{\omega}\right) + V_{\eps}\]
and there exists a constant $C>0$ such that
\[|c_\eps| + \|V_\eps\|_{H^2(K_b)} \leq C\, \|F_{\eps}\|_{L^2(K_b)}.
\]
From (\ref{DefSourceTerm}), we deduce that we have 
\[|c_\eps| + \|V_\eps\|_{H^2(\Omega)} \leq C(\|u_\eps\|_{H^1(K_b)} + \|f\|_{L^2(\Omega)}).\]
From \cite[Theorem 1.4.5.3]{Gris85}, the function $(r,\theta) \mapsto \zeta(r) r^{\pi/\omega} \sin (\pi \theta/\omega)$ belongs to $H^s(\Omega)$ for any $s<1+\pi/\omega$.
We conclude from (\ref{estim1}) that there exists a constant $C>0$ which depends only on the geometry such that, for $s<1+\pi/\omega$,
\[
\left\{
\begin{array}{cl}
\sqrt{\eps}\|\zeta u_\eps\|_{H^2(\Omega)} \leq C\, \|f\|_{L^2(\Omega)}&  \quad\mbox{if } \omega<\pi\\[4pt]
\sqrt{\eps}\|\zeta u_\eps\|_{H^{s}(\Omega)} \leq C\, \|f\|_{L^2(\Omega)}&  \quad\mbox{if } \omega> \pi.
\end{array}
\right.
\]
We remark from (\ref{Pr_cauchy_simple_strong}) that the function $d_\eps=u_\eps-\la_\eps$ satisfies $-\Delta d_\eps=f$ in $\Omega$ and
$\partial_\nu d_\eps=0$ on $\Gamma$, which implies that 
$\zeta d_\eps$ satisfies
problem (\ref{neumann}) with
\[F_{\eps}=-\Delta \zeta d_\eps-2\nabla \zeta\cdot \nabla d_\eps + \zeta f.\]
By using Lemma \ref{Neumann}, we have that
there exists a unique constant $c_\eps \in \mathbb{R}$ and a unique function $V_\eps \in H^2(K_b)$
such that
\[\zeta d_\eps= c\, r^{\pi/\omega} \cos\left(\frac{\pi \theta}{\omega}\right) + V_{\eps}\]
and there exists a constant $C>0$ such that
\[|c_\eps| + \|V_\eps\|_{H^2(K_b)} \leq C\, \|F_{\eps}\|_{L^2(K_b)}.\]
We infer that
\[|c_\eps| + \|V_\eps\|_{H^2(K_b)} \leq C(\|d_\eps\|_{H^1(K_b)} + \|f\|_{L^2(K_b)}) \leq C(\|u_\eps\|_{H^1(K_b)} + \|\la_\eps\|_{H^1(K_b)} + \|f\|_{L^2(K_b)}).\]
And we conclude from (\ref{estim1}) that there exists a constant $C>0$ which depends only on the geometry such that
\[
\left\{
\begin{array}{cl}
\sqrt{\eps}\|\zeta d_\eps\|_{H^2(\Omega)} \leq C\, \|f\|_{L^2(\Omega)}&   \quad\mbox{if } \omega<\pi\\[4pt]
\sqrt{\eps}\|\zeta d_\eps\|_{H^{s}(\Omega)} \leq C\, \|f\|_{L^2(\Omega)}&  \quad\mbox{if } \omega> \pi,
\end{array}
\right.
\]
so that $\la_\eps=d_\eps-u_\eps$ satisfies the same estimate.
The case when $f$ is such that there is a solution $u$ to (\ref{cauchy_strong_simple}) follows the same lines: it suffices to use estimate (\ref{estim2}) instead of (\ref{estim1}).
\end{proof}
\subsection{Regularity at a corner of type $\tilde{\Gamma}$}\label{CornerGammaTilde}
We reuse the notations introduced in the last section.
\begin{proposition}
\label{Gamma_tilde}
Assume that $S$ is the vertex of a corner of type $\tilde{\Gamma}$.
Let us consider $s <1+\pi/\omega$ if $\omega>\pi$ and $s=2$ otherwise. 
For $f \in L^2(\Omega)$, the solution $(u_\eps,\la_\eps) \in V_{0} \times \tilde{V}_0$ to the problem (\ref{Pr_cauchy_simple}) is such that $\zeta u_\eps$ and $\zeta \la_\eps$ belong to $H^{s}(\Omega)$ and there exists a constant $C>0$ which depends only on the geometry such that
\[\forall \eps \in (0,1], \quad \eps \|\zeta u_\eps\|_{H^{s}(\Omega)} + \|\zeta \la_\eps\|_{H^{s}(\Omega)} \leq C \|f\|_{L^2(\Omega)}.\]
If in addition we assume that $f$ is such that problem (\ref{cauchy_strong_simple}) has a (unique) solution $u$, then
\[\forall \eps \in (0,1], \quad \sqrt{\eps} \|\zeta u_\eps\|_{H^{s}(\Omega)} + \frac{\|\zeta \la_\eps\|_{H^{s}(\Omega)}}{\sqrt{\eps}} \leq C \|u\|_{H^1(\Delta,\Omega)}.\]
\end{proposition}
\begin{proof} 
From (\ref{Pr_cauchy_simple_strong}) we have that $\zeta \la_\eps$ satisfies
problem (\ref{dirichlet}) with
\[F_{\eps}=-\Delta \zeta \la_\eps-2\nabla \zeta\cdot \nabla \la_\eps - \eps \zeta \frac{f}{1+\eps}.\]
By using Lemma \ref{Dirichlet}, we have that
there exists a unique constant $c_\eps \in \mathbb{R}$ and a unique function $V_\eps \in H^2(K_b)$
such that
\[\zeta \la_\eps=c_{\eps}\, r^{\pi/\omega} \sin\left(\frac{\pi \theta}{\omega}\right) + V_{\eps}\]
and there exists a constant $C>0$ such that
\[|c_\eps| + \|V_\eps\|_{H^2(K_b)} \leq C\, \|F_{\eps}\|_{L^2(K_b)}.\]
We deduce the estimate
\[|c_\eps| + \|V_\eps\|_{H^2(K_b)} \leq C(\|\la_\eps\|_{H^1(\Omega)} +\eps \|f\|_{L^2(\Omega)}).\]
And we conclude from (\ref{estim1}) that there exists a constant $C>0$ which depends only on the geometry such that, for $s<1+\pi/\omega$
\[
\left\{
\begin{array}{cl}
\|\zeta \la_\eps\|_{H^2(\Omega)} \leq C\, \|f\|_{L^2(\Omega)}&  \quad\mbox{ if }\omega<\pi\\[4pt]
\|\zeta \la_\eps\|_{H^{s}(\Omega)} \leq C\, \|f\|_{L^2(\Omega)}&  \quad\mbox{ if }\omega> \pi.\\
\end{array}
\right.
\]
We remark from (\ref{Pr_cauchy_simple_strong}) that the function $s_\eps=\eps u_\eps+ \la_\eps$ satisfies $-\Delta s_\eps=0$ in $\Omega$ and
$\partial_\nu s_\eps=0$ on $\tilde{\Gamma}$, which implies that 
$\zeta s_\eps$ satisfies
problem (\ref{neumann}) with
\[F_{\eps}=-\Delta \zeta s_\eps-2\nabla \zeta\cdot \nabla s_\eps.\]
By using Lemma \ref{Neumann}, we have that
there exists a unique constant $c_\eps \in \mathbb{R}$ and a unique function $V_\eps \in H^2(K_b)$
such that
\[\zeta s_\eps=c_{\eps}\,r^{\pi/\omega} \cos\left(\frac{\pi \theta}{\omega}\right) + V_{\eps}\]
and there exists a constant $C>0$ such that
\[|c_\eps| + \|V_\eps\|_{H^2(K_b)} \leq C\, \|F_{\eps}\|_{L^2(K_b)}.\]
We infer that
\[|c_\eps| + \|V_\eps\|_{H^2(K_b)} \leq C\|s_\eps\|_{H^1(K_b)} \leq C(\eps \|u_\eps\|_{H^1(K_b)} + \|\la_\eps\|_{H^1(K_b)}).\]
And we conclude from (\ref{estim1}) that there is a constant $C>0$ which depends only on $\Om$ such that
\[
\left\{
\begin{array}{cl}
\|\zeta s_\eps\|_{H^2(\Omega)} \leq C\, \|f\|_{L^2(\Omega)}&  \quad\mbox{if }\omega<\pi\\[4pt]
\|\zeta s_\eps\|_{H^{s}(\Omega)} \leq C\, \|f\|_{L^2(\Omega)}&  \quad\mbox{if }\omega> \pi,\\
\end{array}
\right.
\]
so that $\eps u_\eps=s_\eps-\la_\eps$ satisfies the same estimate.
The case when $f$ is such that there is a solution $u$ to (\ref{cauchy_strong_simple}) is similar.
\end{proof}
\begin{remark}
We emphasize that the small parameter $\eps$ plays a different role in Proposition \ref{Gamma} and in Proposition \ref{Gamma_tilde}.
In Proposition \ref{Gamma}, the exponent in $\eps$ is the same before $u_\eps$ and before $\la_\eps$ because the corner is of type $\Gamma$ and $\Gamma$ is the support of the Cauchy data. In proposition \ref{Gamma_tilde}, the exponent in $\eps$ before $u_\eps$ is one more than the one before $\la_\eps$ because 
the corner is of type $\tilde{\Gamma}$ and data on $\tilde{\Gamma}$ are unknown.
\end{remark}
\noindent It remains to analyze the regularity of functions $u_\eps$ and $\la_\eps$ at corners of mixed type and to derive corresponding estimates. As we will see, this is a much more difficult task. The main reason is that  we do not know whether or not the eigenvectors of a certain symbol $\mathscr{L}_\eps$ defined on $(0,\om)$ (see (\ref{DefSymbol})) form a Hilbert basis of $L^2(0,\omega) \times L^2(0,\omega)$. To bypass this difficulty, we will apply the Kondratiev approach of the seminal article \cite{Kond67} (see also \cite{MaPl77,NaPl94,KoMR97,KoMR01} for more recent presentations). We will follow strictly the methodology proposed in these works. However, we emphasize that in our study we have to keep track of the dependence in $\eps$ in all the estimates. This is the reason why we present the procedure in details. Let us mention that a somehow similar analysis has been conducted in a simpler situation in \cite[Annex]{ChCi13}. We start by presenting some preliminaries on weighted Sobolev spaces borrowed from \cite{KoMR97}.
\section{Some preliminaries on weighted Sobolev spaces}\label{SectionWeightedSobo}
Let us consider the strip $B=\{(t,\theta) \in \mathbb{R}  \times (0,\omega)\}$ for $\omega>0$.
For $\beta \in \mathbb{R}$ and $m \in \mathbb{N}$, let us introduce the weighted Sobolev space 
\[W_\beta^m(B)=\{v \in L^2_{\rm loc}(B),\, e^{\beta t} v \in H^m(B)\},\]
equipped with the norm
\be \|v\|_{W_\beta^m(B)}=\|e^{\beta t} v\|_{H^m(B)}.\label{Wm}\ee
We also denote $\mathring{W}_\beta^m(B)$
the closure of $\mathscr{C}^{\infty}_0(B)$ in $W_\beta^m(B)$, $\mathring{W}_{\beta,0}^m(B)$ the closure in $W_\beta^m(B)$ of the set of functions in $\mathscr{C}^{\infty}_0(\overline{B})$ which vanish in a vicinity of 
$\partial B_0=\partial B \cap \{\theta=0\}$, $\mathring{W}_{\beta,\omega}^m(B)$ the closure in $W_\beta^m(B)$ of the set of functions in $\mathscr{C}^{\infty}_0(\overline{B})$ which vanish in a vicinity of $\partial B_\omega=\partial B \cap \{\theta=\omega\}$.
Let us introduce the Laplace transform 
\be \widehat{v}(\la)=(\mathcal{L}v)(\la)=\int_{-\infty}^{+\infty} e^{-\la t}v(t)\,dt. \label{laplace}\ee
We recall the following properties of the Laplace transform.
\begin{enumerate}
\item
The Laplace transform is a linear and continuous map from $\mathscr{C}^{\infty}_0(\mathbb{R})$ to the space of holomorphic functions in the complex plane.
In addition, we have $\mathcal{L}(\partial_t v)=\la \mathcal{L}(v)$ for all $v \in \mathscr{C}^{\infty}_0(\mathbb{R})$.
\item Pour all $u,v \in \mathscr{C}^{\infty}_0(\mathbb{R})$, we have the Parseval identity
\[\int_{-\infty}^{+\infty}e^{2\beta t}u(t)\overline{v(t)}\,dt=\frac{1}{2\pi i}\int_{{\rm Re}\, \la=-\beta} \widehat{u}(\la)\overline{\widehat{v}(\la)}\,d\la.\]
Hence, the Laplace transform (\ref{laplace}) can be extended as an isomorphism from $L^2_\beta(\mathbb{R})$ to $L^2(\ell_{-\beta})$, where 
$L^2_\beta(\mathbb{R})=\{v \in L^2_{\rm loc}(\mathbb{R}),\, e^{\beta t} v \in L^2(\mathbb{R})\}$ and $\ell_{-\beta}=\{\la=-\beta+ i \tau,\,\tau \in \mathbb{R}\}$.
\item The inverse Laplace transform is given by the formula
\[v(t)=(\mathcal{L}^{-1}\widehat{v})(t)=\frac{1}{2\pi i} \int_{\ell_{-\beta}} e^{\la t}\widehat{v}(\la)\,d\la.\]
\item If $v \in L_{\beta_1}^2(\mathbb{R}) \cap L_{\beta_2}^2(\mathbb{R})$ for $\beta_1<\beta_2$, then the function $\la \mapsto \widehat{v}(\la)$ is holomorphic in the strip defined by
$-\beta_2<{\rm Re}\,\la<-\beta_1$.
\end{enumerate} 
By using the above properties, one can prove that for $\beta \in \mathbb{R}$ and $m \in \mathbb{N}$, the
norm (\ref{Wm}) is equivalent to the norm given
by
\be \|v\|_{\beta,m}=\left(\frac{1}{2\pi i}\int_{\ell_{-\beta}}\|\widehat{v}\|^2_{H^m(0,\omega;\la)}\,d\la\right)^{1/2},\label{norm}\ee
where
\be \|\widehat{v}\|^2_{H^m(0,\omega;\la)}=\|\widehat{v}\|^2_{H^m(0,\omega)} + |\la|^{2m} \|\widehat{v}\|^2_{L^2(0,\omega)}.\label{norm_param}\ee
Next, we introduce the infinite cone
\[K=\{(r\cos \theta,r\sin\theta),\,r>0,\,0<\theta<\omega\},\]
with $\omega \in (0,2\pi)$.
For $\beta \in \mathbb{R}$ and $m \in \mathbb{N}$, let us introduce the weighted Sobolev space 
$V_\beta^m(K)$ as the closure of $\mathscr{C}^{\infty}_0(\overline{K}\setminus \{0\})$
for the norm
\be \|v\|_{V_\beta^m(K)}=\left(\sum_{|\alpha| \leq m} \|r^{|\alpha|-m+\beta} \partial_x^{\alpha}v\|_{L^2(K)}^2\right)^{1/2}.\label{Vm}\ee
We also denote by $\mathring{V}_\beta^m(K)$ the closure of $\mathscr{C}^{\infty}_0(K)$ in $V_\beta^m(K)$, $\mathring{V}_{\beta,0}^m(K)$ the closure in $V_\beta^m(K)$ of the set of functions in $\mathscr{C}^{\infty}_0(\overline{K})$ which vanish in a vicinity of $\partial K_0=\partial K \cap \{\theta=0\}$, $\mathring{V}_{\beta,\omega}^m(K)$ the closure in $V_\beta^m(K)$ of the set of functions in $\mathscr{C}^{\infty}_0(\overline{K})$ which vanish in a vicinity of $\partial K_\omega=\partial K \cap \{\theta=\omega\}$.
One can show that the norm of $V_\beta^m(K)$ is equivalent to the norm
\be \|v\|=\left(\int_0^{+\infty}\ r^{2(-m+\beta)} \sum_{j=0}^m \|(r\partial_r)^j v(r,\cdot)\|^2_{H^{m-j}(0,\omega)}r\,dr\right)^{1/2}.\label{normK}\ee
The key point consists in the change of variable $t=\ln r$, which transforms the cone $K=\mathbb{R}^*_+ \times (0,\omega)$ into the strip $B=\mathbb{R} \times (0,\omega)$.
In particular, if we introduce, for a function $v$ defined in $K$, the function $\mathcal{E}v$ defined in $B$ by
\[(\mathcal{E}v)(\theta,t)=v(\theta,e^t),\]
since $r\partial_r v=\partial_t (\mathcal{E}v)$, the norm (\ref{normK}) is equivalent to 
\[ \|v\|=\left(\int_0^{+\infty}\ e^{2(-m+\beta+1)t} \sum_{j=0}^m \|\partial_t^j (\mathcal{E}v)(t,\cdot)\|^2_{H^{m-j}(0,\omega)}\,dt\right)^{1/2},\]
hence
\[\|v\|=\|e^{(-m+\beta+1)t}\mathcal{E}v\|_{H^m(B)}=\|\mathcal{E}v\|_{W^m_{\beta-m+1}(B)}.\]
This shows that there exists an isomorphism between the spaces $V_\beta^m(K)$ and $W^m_{\beta-m+1}(B)$, or in other words, between $W_\beta^m(B)$ and $V^m_{\beta+m-1}(K)$.
\section{The case of a corner of mixed type}\label{CornerGammaGammaTilde}
The regularity of solutions $u_\eps$ and $\la_\eps$ at a corner of mixed type can no longer be analyzed separately.
We use the weighted Sobolev spaces introduced in the previous section.
We first consider the quasi-reversibility problem in the strip $B$.
The strong equations corresponding to (\ref{Pr_cauchy_simple}) in the strip are
\be \left\{
\begin{array}{ll}
\displaystyle  -\Delta u_\eps  = \Delta \la_\eps/\eps= f/(1+\eps)&  \mbox{in }B\\[4pt]
\displaystyle u_\eps=0 \,\,\text{and}\,\,\partial_\nu u_\eps -\partial_\nu \la_\eps=0 & \mbox{on }\partial B_0 \\[4pt]
\displaystyle \la_\eps=0\,\,\text{and}\,\, \eps\, \partial_\nu u_\eps + \partial_\nu \la_\eps=0& \mbox{on }\partial B_\omega. 
\end{array}
\right.
\label{Pr_cauchy_strong_B}
\ee
For $\beta\in\R$, define the operator $\mathcal{B}_\beta: \mathcal{D}(\mathcal{B}_\beta)\rightarrow \mathcal{R}(\mathcal{B}_\beta)$ such that $\mathcal{B}_\beta(u_\eps,\la_\eps)=(f_1,f_2)$,
with 
\begin{equation}\label{defOpBbeta}
\begin{array}{ll}
&(f_1,f_2)=(-\Delta u_\eps,-\Delta \la_\eps/\eps)\\[3pt]
&\mathcal{D}(\mathcal{B}_\beta)=\{(u_\eps,\la_\eps) \in \mathring{W}^1_{\beta,0}(B) \cap W_\beta^{2}(B) \times \mathring{W}^1_{\beta,\omega}(B) \cap  W_\beta^{2}(B),\\[3pt]
&\hspace{2cm}\partial_\nu u_\eps-\partial_\nu \la_\eps=0\,\,{\rm on}\,\, \partial B_0,\quad \eps \partial_\nu u_\eps+ \partial_\nu \la_\eps =0 \,\,{\rm on}\,\, \partial B_\omega\}\\[3pt]
\mbox{and }\qquad&\mathcal{R}(\mathcal{B}_\beta)=W^0_\beta(B) \times W_\beta^0(B).
\end{array}
\end{equation}
This operator is associated with the following problem in the strip $B$:
\be \left\{
\begin{array}{rcll}
 \displaystyle -\Delta u_\eps  &=& f_1&  \mbox{in } B\\
 \displaystyle -\Delta \la_\eps  &=& \eps\,f_2&  \mbox{in } B\\
 u_\eps&=&0  & \mbox{on } \partial B_0 \\
 \partial_\nu u_\eps -\partial_\nu \la_\eps&=&0 & \mbox{on } \partial B_0 \\
 \la_\eps&=&0  & \mbox{on } \partial B_\omega \\
 \eps\, \partial_\nu u_\eps + \partial_\nu \la_\eps&=&0 & \mbox{on }\partial B_\omega.
\end{array}
\right.
\label{Pr_cauchy_strip}
\ee
If we apply the Laplace transform to problem (\ref{Pr_cauchy_strip}), the following symbol $\mathscr{L}_\eps(\la): \mathcal{D}(\mathscr{L}_\eps) \rightarrow \mathcal{R}(\mathscr{L}_\eps)$ such that $\mathscr{L}_\eps(\la)(\varphi_\eps,\psi_\eps)=(g_1,g_2)$ naturally appears, with
\begin{equation}\label{DefSymbol}
\begin{array}{llcl}
&(g_1,g_2)&=&\dsp(-(\la^2+d^2_\theta)\varphi_\eps,-\frac{1}{\eps}(\la^2+d^2_\theta)\psi_\eps)\\[3pt]
&\mathcal{D}(\mathscr{L}_\eps)&=&\{(\varphi_\eps,\psi_\eps) \in H^{2}(0,\omega) \times H^{2}(0,\omega),\,\varphi_\eps(0)=0,\, \psi_\eps(\omega)=0,\\[3pt]
&&&\hspace{0.2cm}d_\theta \varphi_\eps(0)-d_\theta \psi_\eps(0)=0,\, \eps d_\theta \varphi_\eps(\omega)+ d_\theta \psi_\eps(\omega) =0\},\\[3pt]
\mbox{ and }&\mathcal{R}(\mathscr{L}_\eps)&=&L^2(0,\omega) \times L^2(0,\omega).
\end{array}
\end{equation}
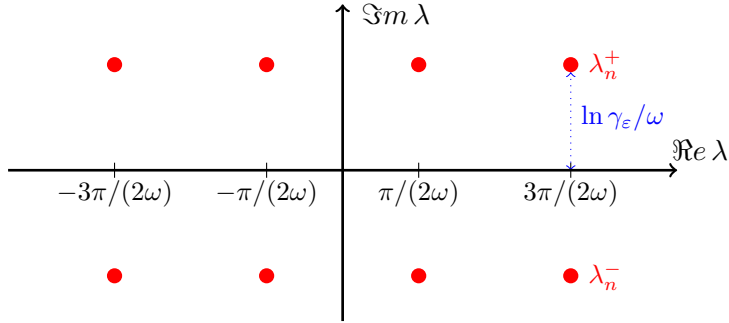
\begin{figure}[!ht]
\centering
\begin{tikzpicture}
\draw (1,-0.1)--(1,0.1);
\draw (3,-0.1)--(3,0.1);
\draw (-1,-0.1)--(-1,0.1);
\draw (-3,-0.1)--(-3,0.1);
\draw[<->,dotted,blue] (3,0)--(3,1.3);
\node at (1,0) [anchor=north]{\small $\pi/(2\omega)$};
\node at (3,0) [anchor=north]{\small $3\pi/(2\omega)$};
\node at (-1,0) [anchor=north]{\small $-\pi/(2\omega)$};
\node at (-3,0) [anchor=north]{\small $-3\pi/(2\omega)$};
\node at (3,0.7) [anchor=west]{\small \textcolor{blue}{$\ln \gamma_\eps/\omega$}};
\node at (3.1,1.4) [anchor=west]{\small \textcolor{red}{$\lambda^+_n$}};
\node at (3.1,-1.4) [anchor=west]{\small \textcolor{red}{$\lambda^-_n$}};
\draw[draw=black,line width=1pt,->](-4.4,0)--(4.4,0);
\draw[draw=black,line width=1pt,->](0,-2)--(0,2.2);
\filldraw [red,draw=none] (1,1.4) circle (0.1);
\filldraw [red,draw=none] (1,-1.4) circle (0.1);
\filldraw [red,draw=none] (-1,1.4) circle (0.1);
\filldraw [red,draw=none] (-1,-1.4) circle (0.1);
\filldraw [red,draw=none] (3,1.4) circle (0.1);
\filldraw [red,draw=none] (3,-1.4) circle (0.1);
\filldraw [red,draw=none] (-3,1.4) circle (0.1);
\filldraw [red,draw=none] (-3,-1.4) circle (0.1);
\node at (4.7,0) [anchor=south] {$\Re e\,\lambda$};
\node at (0.7,1.8) [anchor=south] {$\Im m\,\lambda$};
\end{tikzpicture}
\caption{Position of the $\lambda_n^{\pm}$ in the complex plane.\label{fig local vp}}
\end{figure}
We will say that $\la \in \mathbb{C}$ is an eigenvalue of $\mathscr{L}_\eps$ if ${\rm Ker}\,\mathscr{L}_\eps(\la) \neq \{0\}$.
We have the following lemma.
\begin{lemma}
\label{eigenvalue_qr}
The eigenvalues of the symbol $\mathscr{L}_\eps$ are 
\[\la^\pm_n=\frac{1}{\omega}\left(\frac{\pi}{2}+n\pi \pm i \ln \gamma_\eps\right),\quad n \in \mathbb{Z},\]
with 
\[\gamma_\eps=\sqrt{1+\frac{1}{\eps}}+\sqrt{\frac{1}{\eps}}\qquad\mbox{(see Figure \ref{fig local vp})}.\]
The corresponding (non normalized) eigenfunctions are given by
\[\varphi^\pm_n(\theta)=\cos(\la^\pm_n \omega)\sin(\la^\pm_n \theta),\qquad\qquad \psi^\pm_n(\theta)=\sin(\la^\pm_n(\theta-\omega)).\]
\end{lemma}
\begin{proof}
Let us find all non vanishing pairs $(\varphi,\psi)$ such that
$-(\la^2+d^2_\theta)\varphi=0$ and $-(\la^2+d^2_\theta)\psi=0$
in $(0,\omega)$
with $\varphi(0)=0$, $\psi(\omega)=0$, $d_\theta\varphi(0)-d_\theta \psi(0)=0$ and $\eps d_\theta \varphi(\omega)+d_\theta \psi(\omega)=0$.
It is readily seen that $\la=0$ is not an eigenvalue, so that we assume that $\la \neq 0$ in the sequel.
From the two equations in $(0,\omega)$ and the two first boundary conditions, we obtain that
\[\varphi(\theta)=A \sin(\la \theta),\qquad\qquad \psi(\theta)=B\sin(\la(\theta-\omega)).\]
Then we use the two last boundary conditions, and we obtain (since $\la \neq 0$)
\[A-B \cos (\la \omega)=0,\qquad\qquad \eps A \cos(\la \omega) + B=0.\]
The complex number $\la$ is an eigenvalue if and only if 
\[1+\eps \cos^2(\la\omega)=0,\] that is if and only if
$\cos(\la\omega)=\pm i/\sqrt \eps$. Hence we deduce that we must have
\[z^2+2 \frac{\pm i}{\sqrt{\eps}}z+1=0\qquad \mbox{ for }z:=e^{i\la\omega}.\]
The solutions to these two equations are
$z= \pm i \gamma^\pm_\eps$, with 
\[\gamma^\pm_\eps=\sqrt{1+\frac{1}{\eps}} \pm \sqrt{\frac{1}{\eps}}.\]
It remains to find $\la$ such that $e^{i\la\omega}=\pm i \gamma^\pm_\eps$.
Writing $\la=a+ib$ with $(a,b) \in \mathbb{R}^2$, since $\pm i =e^{i\pi(1\mp 1/2)}$, we find
\[e^{-b \omega}=\gamma^\pm_\eps,\qquad a\omega=\pi(1 \mp \frac{1}{2})+2n\pi, \quad n \in \mathbb{Z}.\]
This implies
\[b \omega=-\ln \gamma^\pm_\eps,\qquad\qquad a\omega= \frac{\pi}{2}+n\pi, \quad n \in \mathbb{Z},\]
which gives the result, in view of $\ln \gamma^-_\eps=-\ln \gamma^+_\eps$ (note that $\gamma^+_\eps\gamma^-_\eps=1$).
\end{proof}
\begin{remark}
We notice that the symbol $\mathscr{L}_\eps$ has complex eigenvalues and is not self-adjoint. This is a difference with the symbols which are involved when considering the Laplace equation with Dirichlet or Neumann boundary conditions.
\end{remark}
\noindent Let us first consider the case $\beta=0$. Then we simply denote $\mathring{W}_{0,0}^1(B)=H^1_{0,0}(B)$ and $\mathring{W}_{0,\omega}^1(B)=H^1_{0,\omega}(B)$.
We have the following theorem.
\begin{theorem}\label{th_0qr}
The operator $\mathcal{B}_0$ defined in (\ref{defOpBbeta}) is an isomorphism. Furthermore, there exists a constant $C>0$ such that for all $(f_1,f_2) \in \mathcal{R}(\mathcal{B}_0)$, the solution $(u_\eps,\la_\eps) \in \mathcal{D}(\mathcal{B}_0)$ to the problem (\ref{Pr_cauchy_strong_B}) satisfies
\[\sqrt{\eps} \|u_\eps\|_{H^2(B)} + \|\la_\eps\|_{H^2(B)} \leq C (\|f_1\|_{L^2(B)}+ \sqrt{\eps}\|f_2\|_{L^2(B)}).\]
\end{theorem}
\begin{proof}
We simply have
\[
\begin{array}{ll}
&\mathcal{D}(\mathcal{B}_0)=\{(u_\eps,\la_\eps) \in H^1_{0,0}(B) \cap H^{2}(B) \times H^1_{0,\omega}(B) \cap H^{2}(B),\\[3pt]
&\hspace{2cm}\partial_\nu u_\eps-\partial_\nu \la_\eps=0\,\,{\rm on}\,\, \partial B_0,\quad \eps \partial_\nu u_\eps+ \partial_\nu \la_\eps =0 \,\,{\rm on}\,\, \partial B_\omega\}\\[3pt]
\mbox{and }&\mathcal{R}(\mathcal{B}_0)=L^2(B) \times L^2(B).
\end{array}
\]
By applying the Laplace transform to the problem (\ref{Pr_cauchy_strip}) with respect to $t$ and by setting $\la=i\tau$ with $\tau \in \mathbb{R}$, we obtain
\[ \left\{
\begin{array}{rcll}
\displaystyle (\tau^2 - d^2_\theta) \widehat{u}_\eps  &=& \widehat{f}_1 &\mbox{ in } (0,\omega)\\
\displaystyle (\tau^2 - d^2_\theta)\widehat{\la}_\eps  &=& \eps \widehat{f}_2 &\mbox{ in } (0,\omega)\\
\widehat{u}_\eps(0)&=&0&  \\
d_\theta \widehat{u}_\eps(0) -d_\theta \widehat{\la}_\eps(0)&=&0& \\
\widehat{\la}_\eps(\omega)&=&0&  \\
\eps\, d_\theta \widehat{u}_\eps(\omega) + d_\theta \widehat{\la}_\eps(\omega)&=&0.&
\end{array}
\right.
\]
For fixed $\tau$, this problem is equivalent to the weak formulation: find $(\widehat{u}_\eps,\widehat{\la}_\eps) \in H_{0,0}^1(0,\omega) \times H_{0,\omega}^1(0,\omega)$, where $H_{0,0}^1(0,\omega)=\{v\in H^1(0,\omega),\,\, v(0)=0\}$ and 
$H_{0,\omega}^1(0,\omega)=\{\mu \in H^1(0,\omega),\,\, \mu(\omega)=0\}$, such that for all $(v,\mu) \in H_{0,0}^1(0,\omega) \times H_{0,\omega}^1(0,\omega)$,
\begin{equation} 
\left\{\begin{array}{rcl} 
\displaystyle \eps \int_0^\omega (d_\theta \widehat{u}_\eps\, d_\theta \overline{v} + \tau^2 \widehat{u}_\eps \overline{v})\,d\theta   + \int_0^\omega (d_\theta \overline{v}\, d_\theta \widehat{\lambda}_\eps+\tau^2 \overline{v}\, \widehat{\la}_\eps)\,d\theta&=& \dsp\int_0^\omega \eps (\widehat{f}_1+ \widehat{f}_2)\overline{v}\,d\theta\\[8pt]
\displaystyle \int_0^\omega (d_\theta \widehat{u}_\eps\, d_\theta \overline{\mu} +\tau^2 \widehat{u}_\eps \overline{\mu})\,d\theta - \int_0^\omega (d_\theta \widehat{\lambda}_\eps\,d_\theta \overline{\mu}+\tau^2 \widehat{\la}_\eps \overline{\mu})\,d\theta&=&\dsp
\int_0^\omega (\widehat{f}_1- \eps \widehat{f}_2)\overline{\mu}\,d\theta.
\end{array} \right. 
\label{weak_1d} 
\end{equation}
By the Lax-Milgram Lemma, the weak formulation (\ref{weak_1d}) is well-posed and there exists
some constant $C>0$ (independent of $\la$ and of $\eps$) such that
\be \label{estimate_qr}
\begin{array}{l}
\dsp\sqrt{\eps} (\|\widehat{u}_\eps\|_{H^2(0,\omega)}+ |\la|^2 \|\widehat{u}_\eps\|_{L^2(0,\omega)})
+ \|\widehat{\la}_\eps\|_{H^2(0,\omega)}+ |\la|^2 \|\widehat{\la}_\eps\|_{L^2(0,\omega)}\\[5pt]
\dsp\leq C(\|\widehat{f}_1\|_{L^2(0,\omega)} + \sqrt{\eps}\|\widehat{f}_2\|_{L^2(0,\omega)}). 
\end{array}
\ee
Indeed, by setting $v=\widehat{u}_\eps$ and $\mu=\widehat{\la}_\eps$ in (\ref{weak_1d}), we obtain
\[
\begin{array}{ll}
&\dsp\eps (\|d_\theta \widehat{u}_\eps\|^2_{L^2(0,\omega)} + |\la|^2 \|\widehat{u}_\eps\|^2_{L^2(0,\omega)})
+ \|d_\theta \widehat{\la}_\eps\|^2_{L^2(0,\omega)} + |\la|^2 \|\widehat{\la}_\eps\|^2_{L^2(0,\omega)}\\[8pt]
=&\dsp\int_0^\omega \eps (\widehat{f}_1+ \widehat{f}_2)\overline{\widehat{u}_\eps}\,d\theta-\int_0^\omega (\overline{\widehat{f}_1}- \eps \overline{\widehat{f}_2})\widehat{\la}_\eps\,d\theta\\[8pt]
\leq & \dsp(\eps \|\widehat{f}_1+\widehat{f}_2\|_{L^2(0,\omega)}^2+ \|\widehat{f}_1-\eps\widehat{f}_2\|_{L^2(0,\omega)}^2)^{1/2} (\eps \|\widehat{u}_\eps\|_{L^2(0,\omega)}^2+ \|\widehat{\la}_\eps\|_{L^2(0,\omega)}^2)^{1/2}\\[6pt]
\leq & \sqrt{1+\eps}\,(\|\widehat{f}_1\|_{L^2(0,\omega)} + \sqrt{\eps}\|\widehat{f}_2\|_{L^2(0,\omega)})(\eps \|\widehat{u}_\eps\|^2_{L^2(0,\omega)} + \|\widehat{\la}_\eps\|^2_{L^2(0,\omega)}))^{1/2}.
\end{array}
\]
By using the Poincar\'e inequality and assuming that $\eps \leq 1$ we obtain that
\[\sqrt{\eps} \|\widehat{u}_\eps\|_{H^1(0,\omega)} + \|\widehat{\la}_\eps\|_{H^1(0,\omega)} \leq C(\|\widehat{f}_1\|_{L^2(0,\omega)} + \sqrt{\eps} \|\widehat{f}_2\|_{L^2(0,\omega)}) \]
and
\[|\la|^2(\sqrt{\eps} \|\widehat{u}_\eps\|_{L^2(0,\omega)} + \|\widehat{\la}_\eps\|_{L^2(0,\omega)}) \leq C(\|\widehat{f}_1\|_{L^2(0,\omega)} + \sqrt{\eps} \|\widehat{f}_2\|_{L^2(0,\omega)}),\]
where $C$ is independent of $\la$ and $\eps$.
Now, given that
\[d^2_\theta \widehat{u}_\eps  = |\la|^2 \widehat{u}_\eps -\widehat{f}_1,\qquad\qquad
d^2_\theta \widehat{\la}_\eps  = |\la|^2 \widehat{\la}_\eps -\eps\widehat{f}_2,\]
we deduce
\[
\begin{array}{ll}
&\sqrt{\eps}\|d^2_\theta \widehat{u}_\eps\|_{L^2(0,\omega)}+\|d^2_\theta \widehat{\la}_\eps\|_{L^2(0,\omega)}\\[2pt]
\leq &|\la|^2(\sqrt{\eps} \|\widehat{u}_\eps\|_{L^2(0,\omega)} + \|\widehat{\la}_\eps\|_{L^2(0,\omega)})+ C (\|\widehat{f}_1\|_{L^2(0,\omega)} +\sqrt{\eps} \|\widehat{f}_2\|_{L^2(0,\omega)})\\[2pt]
\leq &C(\|\widehat{f}_1\|_{L^2(0,\omega)} + \sqrt{\eps}\|\widehat{f}_2\|_{L^2(0,\omega)}),
\end{array}
\]
which implies (\ref{estimate_qr}). Finally, we have for all $\la=i\tau$
\[\sqrt{\eps} \|\widehat{u}_\eps(\la,\cdot)\|_{H^2(0,\omega,\la)} \leq C(\|\widehat{f}_1\|_{L^2(0,\omega)} + \sqrt{\eps} \|\widehat{f}_2\|_{L^2(0,\omega)})\]
and
\[ \|\widehat{\la}_\eps(\la,\cdot)\|_{H^2(0,\omega,\la)} \leq C(\|\widehat{f}_1\|_{L^2(0,\omega)} + \sqrt{\eps}\|\widehat{f}_2\|_{L^2(0,\omega)}),\]
which by integration on $\ell_{0}$ and by definition of the norms $\|\cdot\|_{\beta,m}$ (see (\ref{norm})) implies
\[\sqrt{\eps} \|u_\eps\|_{0,2} + \|\la_\eps\|_{0,2} \leq C (\|f_1\|_{0,0}+\sqrt{\eps}\|f_2\|_{0,0}).\]
This gives the estimate
\[\sqrt{\eps} \|u_\eps\|_{H^2(B)} + \|\la_\eps\|_{H^2(B)} \leq C (\|f_1\|_{L^2(B)}+ \sqrt{\eps}\|f_2\|_{L^2(B)})\]
which proves that $\mathcal{B}_0$ is an isomorphism.
\end{proof}
\noindent Now we wish to extend the result of Theorem \ref{th_0qr} to any $\beta \notin \{(\pi/2+n\pi)/\om,\,n\in\mathbb{Z}\}$.
\begin{theorem}
\label{th_betaqr}
For any $\beta \notin \{(\pi/2+n\pi)/\om,\,n\in\mathbb{Z}\}$, the operator $\mathcal{B}_\beta$ is an isomorphism.
Furthermore, there exists a constant $C>0$ independent of $\eps$ such that for all
$(f_1,f_2) \in \mathcal{R}(\mathcal{B}_\beta)$, the solution $(u_\eps,\la_\eps) \in \mathcal{D}(\mathcal{B}_\beta)$ to the problem 
(\ref{Pr_cauchy_strip}) (which depends on $\beta$) satisfies
\begin{equation}\label{EstimateUniform_B}
\sqrt{\eps} \|u_\eps\|_{\beta,2}+\|\la_\eps\|_{\beta,2} \leq C\,(\|f_1\|_{\beta,0}+\sqrt{\eps}\|f_2\|_{\beta,0}).
\end{equation}
\end{theorem}
\begin{proof}
For $(\hat{f_1},\hat{f_2}) \in L^2(0,\omega) \times L^2(0,\omega)$, we consider the problem of finding the functions
$(\hat{u}_\eps,\hat{\la}_\eps)\in H^1_{0,\,0}(0,\omega)\times H^1_{0,\,
\omega}(0,\omega)$ such that
\begin{equation}\label{PbFort}
\left\{\begin{array}{rcll}
\displaystyle-\eps\,(d^2_\theta+\lambda^2)\hat{u}_\eps &=& \eps\,\hat{f_1} &\mbox{ in } (0,\omega)\\[3pt]
\displaystyle -(d^2_\theta+\lambda^2)\hat{\la}_\eps&=&\eps\,\hat{f_2}  &\mbox{ in } (0,\omega)\\[3pt]
\displaystyle d_{\theta}\hat{u}_\eps(0)-d_{\theta}\hat{\la}_\eps(0)&=&0&\\
\displaystyle \eps\,d_{\theta}\hat{u}_\eps(\omega)+d_{\theta}\hat{\la}_\eps(\omega)&=&0.&
\end{array}
\right.
\end{equation}
We wish to prove that there is a constant $C>0$ such that the solution of Problem (\ref{PbFort}) satisfies
\begin{equation}\label{EstimateUniform}
\begin{array}{l}
\dsp\sqrt{\eps}(\|\hat{u}_\eps\|_{H^2(0,\omega)}+|\lambda|^2\|\hat{u}_\eps\|_{L^2(0,\omega)})+\|\hat{\la}_\eps\|_{H^2(0,\omega)}+|\lambda|^2\|\hat{\la}_\eps\|_{L^2(0,\omega)} \\[6pt]
\dsp \le C\,(\|\hat{f_1}\|_{L^2(0,\omega)}+\sqrt{\eps}\|\hat{f_2}\|_{L^2(0,\omega)})
\end{array}
\end{equation}
for all $\eps>0$, $\lambda \in \ell_{\beta}=\{\gamma\in \mathbb{C},\,{\rm Re}\,\gamma=\beta\}$. Note that $C$ depends on $\beta$ but not on $\eps$, $\lambda\in \ell_{\beta}$.
According to  the analytic Fredholm theorem, we know that problem (\ref{PbFort}) admits a unique solution if and only if the only solution for $(\hat{f_1},\hat{f_2})=(0,0)$ is $(\hat{u}_\eps,\hat{\la}_\eps)=(0,0)$, that is if and only if $\la$ is not an eigenvalue of $\mathscr{L}_\eps$. Lemma 
\ref{eigenvalue_qr} guarantees that for $\beta\notin\{(\pi/2+n\pi)/\om,\,n\in\mathbb{Z}\}$, this is indeed the case for all $\eps>0$, $\lambda \in \ell_{\beta}$. \\
\newline
Estimate (\ref{EstimateUniform}) has already been established for $\beta=0$. Now we assume that $\beta\neq 0$. In order to show (\ref{EstimateUniform}), we work with the decomposition $(\hat{u}_\eps,\hat{\la}_\eps)=(u_0,\lambda_0)+(u_{\sharp},\lambda_{\sharp})$, where $u_0\in H^1_{0,\,0}(0,\omega)$ and $\lambda_0\in H^1_{0,\,\omega}(0,\omega)$ are the functions which solve
\[
\left\{
\begin{array}{c}
-(d^2_\theta+\lambda^2)u_0=\hat{f_1} \mbox{ in } (0,\omega)\\
d_\theta u_0(\omega)=0
\end{array}
\right.
\qquad\mbox{ and }\qquad
\left\{
\begin{array}{c}
-(d^2_\theta+\lambda^2)\lambda_0 = \eps\,\hat{f_2} \mbox{ in } (0,\omega)\\
d_\theta \lambda_0(0) = 0.
\end{array}
\right.
\] 
For these classical problems, by using Proposition \ref{Estimate_appendix} in Appendix A, there is a constant $C>0$ such that 
\begin{equation}\label{EstimateUniform0_part_u}
\|d_{\theta}^2u_0\|_{L^2(0,\omega)}+|\lambda|^2\|u_0\|_{L^2(0,\omega)} \leq C\|\hat{f_1}\|_{L^2(0,\omega)}
\end{equation}
and
\begin{equation}\label{EstimateUniform0_part_phi}
\|d_{\theta}^2\lambda_0\|_{L^2(0,\omega)}+|\lambda|^2\|\lambda_0\|_{L^2(0,\omega)} \leq C\eps \|\hat{f_2}\|_{L^2(0,\omega)}
\end{equation}
for all $\lambda\in\ell_{\beta}$ when $\beta\notin\{(\pi/2+n\pi)/\omega,\,n\in\mathbb{Z}\}$. Here and in what follows, the constant $C>0$ may change from a line to another but is independent of $\eps>0$, $\lambda\in\ell_{\beta}:=\{\gamma\in\mathbb{C},\,{\rm Re}\,\gamma=\beta\}$.\\
\newline
Now, we see that $(u_{\sharp},\lambda_{\sharp})\in H^1_{0,\,0}(0,\omega) \times H^1_{0,\,\omega}(0,\omega)$ satisfies
\begin{equation}\label{PbFortSharp}
\left\{\begin{array}{l}
(d^2_\theta+\lambda^2)u_{\sharp}=0\mbox{ in } (0,\omega)\\[3pt]
(d^2_\theta+\lambda^2)\lambda_{\sharp}=0\mbox{ in } (0,\omega)\\[3pt]
d_{\theta}u_{\sharp}(0)-d_{\theta}\lambda_{\sharp}(0)=-d_{\theta}u_0(0)\\[3pt]
\eps\,d_{\theta}u_{\sharp}(\omega)+d_{\theta}\lambda_{\sharp}(\omega)=-d_{\theta}\lambda_{0}(\omega).
\end{array}
\right.
\end{equation}
Looking for $u_{\sharp}$, $\lambda_{\sharp}$  of the form $u_{\sharp}(\theta)=A\,\sin(\lambda\theta)$, $\lambda_{\sharp}(\theta)=B\,\sin(\lambda(\theta-\omega))$, we find that $A$ and $B$ must solve the problem
\[
\left(\begin{array}{cc}
\lambda & -\lambda\,\cos(\lambda\omega)\\ \eps\,\lambda\,\cos(\lambda\om) & \lambda
\end{array}\right)\left(\begin{array}{c}
A \\ B
\end{array}\right)=\left(\begin{array}{c}
-d_{\theta}u_0(0) \\ -d_{\theta}\lambda_{0}(\omega)
\end{array}\right).
\]
We deduce that 
\[
u_{\sharp}(\theta)=-\cfrac{ d_{\theta}u_0(0)+\cos(\lambda\omega)d_{\theta}\lambda_0(\omega)}{\lambda(1+\eps\,\cos^2(\lambda\omega))}\,\sin(\lambda\theta)\]
and
\[\lambda_{\sharp}(\theta)=\cfrac{\eps\cos(\lambda\omega) d_{\theta}u_0(0)-d_{\theta}\lambda_0(\omega)}{\lambda(1+\eps\,\cos^2(\lambda\omega))}\,\sin(\lambda(\theta-\omega)).
\]
From identity (\ref{RelationSin}) of Appendix B,
we have
$|\sin(\lambda\theta)|^2=(\cosh(2\tau\theta)-\cos(2\beta\theta))/2$, for $\lambda=\beta+i\tau$. We can write
\begin{equation}\label{EstimationInterm1}
\|u_{\sharp}\|_{L^2(I)}^2 =\bigg|\cfrac{ d_{\theta}u_0(0)+\cos(\lambda\omega)d_{\theta}\lambda_0(\omega)}{2\lambda(1+\eps\,\cos^2(\lambda\omega))}\bigg|^2(\tau^{-1}\sinh(2\tau\omega)-\beta^{-1}\sin(2\beta\omega)).
\end{equation}
Since $\beta\ne0$, one can verify that there is $C>0$ such that, for all $\tau\in\mathbb{R}$, we have
\begin{equation}\label{EstimationInterm2Equiv}\beta^{-1}\sin(2\beta\omega)= |\beta|^{-1}\sin(2|\beta|\omega) \leq 2\omega \leq |\tau|^{-1}\sinh(2|\tau|\omega)=\tau^{-1}\sinh(2\tau\omega) \leq C\,e^{2|\tau|\om}/|\lambda|.
\end{equation}
Using (\ref{EstimationInterm2Equiv}) in (\ref{EstimationInterm1}), we obtain
\begin{equation}\label{EstimationInterm2}
|\lambda|^4\|u_{\sharp}\|_{L^2(I)}^2
\le  C\,|\lambda|\,\left(\cfrac{ |d_{\theta}u_0(0)|^2}{|1+\eps\,\cos^2(\lambda\om)|^2}+\cfrac{ \,|\cos(\lambda\om)d_{\theta}\lambda_0(\om)|^2}{|1+\eps\,\cos^2(\lambda\om)|^2}\right)\,e^{2|\tau|\om}.
\end{equation}
Now we explain how to obtain estimates for $|d_{\theta}u_0(0)|$ and $|\cos(\lambda\omega)d_{\theta}\lambda_0(\omega)|$.\\ 
\newline
$\star$ First we multiply the equation $-(d^2_\theta+\lambda^2)u_0=\hat{f_1}$ in $(0,\omega)$ by $\cos(\lambda(\theta-\omega))$ and integrate by parts. This gives us
\[
-d_{\theta}u_0(0)\cos(\lambda\om)=\int_0^\omega d^2_{\theta}u_0\cos(\lambda(\theta-\omega))-u_0d^2_{\theta}(\cos(\lambda(\theta-\omega)))\,d\theta=-\int_0^\omega \hat{f_1}\cos(\lambda(\theta-\omega))\,d\theta
\]
and leads to 
\begin{equation}\label{EstimationInterm3}
|\cos(\lambda\omega)d_{\theta}u_0(0)|^2 \leq \|\hat{f_1}\|^2_{L^2(0,\omega)}\|\cos(\lambda(\theta-\omega))\|^2_{L^2(0,\omega)}.
\end{equation}
An analogous computation to what precedes (\ref{EstimationInterm1}), based on Identity (\ref{RelationCos}) of Appendix B, yields 
\[\|\cos(\lambda(\theta-\omega))\|^2_{L^2(0,\omega)}= (\tau^{-1}\sinh(2\tau\omega)+\beta^{-1}\sin(2\beta\omega))/4.\]
Using the latter result as well as (\ref{EstimationInterm2Equiv}), we get
\begin{equation}\label{EstimationInterm4}
|d_{\theta}u_0(0)|^2 \leq C\,\|\hat{f_1}\|^2_{L^2(I)}\,e^{2|\tau|\omega}/(|\lambda|\,|\cos(\lambda\omega)|^2).
\end{equation}
From Identity (\ref{RelationCos}) and by using the fact that $\beta \notin \{(\pi/2+n\pi)/\omega,\, n \in \mathbb{Z}\}$, one can check that there is a constant $C>0$ such that $e^{2|\tau|\omega}/ |\cos(\lambda\omega)|^2\le C$ for all $\tau\in\mathbb{R}$. We deduce from (\ref{EstimationInterm4}) that
\begin{equation}\label{EstimationInterm5}
|d_{\theta}u_0(0)|^2 \le C\,\|\hat{f_1}\|^2_{L^2(0,\omega)}/|\lambda|.
\end{equation}
$\star$ Now, we provide an estimate for $|\cos(\lambda\omega)d_{\theta}\lambda_0(\omega)|$. Multiplying the equation $-(d^2_\theta+\lambda^2)\lambda_0=\eps\,\hat{f_2}$ in $(0,\omega)$ by $\cos(\lambda\theta)$ and integrating by parts, we find
\[
d_{\theta}\lambda_0(\omega)\cos(\lambda\om)=\int_0^\omega d^2_{\theta}\lambda_0\cos(\lambda\theta)-\lambda_0 d^2_{\theta}(\cos(\lambda\theta))\,d\theta=-\eps\int_0^\omega\hat{f_2}\cos(\lambda\theta)\,d\theta.
\]
Working as above, this allows us to write
\begin{equation}\label{EstimationInterm4Bis}
|d_{\theta}\lambda_0(\omega)\cos(\lambda\omega)|^2\le C\,\eps^2\,\|\hat{f_2}\|^2_{L^2(0,\omega)}e^{2|\tau|\omega}/|\lambda|.
\end{equation}
In Lemmas \ref{lemmaEstimate} and \ref{lemmaEstimate2} proved in Appendix B, we get the following estimates
\begin{equation}\label{Estimations}
\cfrac{e^{2|\tau|\omega}}{|1+\eps\,\cos^2(\lambda\omega)|^2}\leq C/\eps, \qquad\qquad \frac{\eps^2 e^{4|\tau|\omega}}{|1+\eps\,\cos^2(\lambda\omega)|^2} \leq C,
\end{equation}
where again $C>0$ is independent of $\eps>0$, $\lambda=\beta+i\tau\in\ell_{\beta}$. Therefore, inserting (\ref{EstimationInterm4}) as well as (\ref{EstimationInterm4Bis}) in (\ref{EstimationInterm2}) and using (\ref{Estimations}), 
we obtain $\sqrt{\eps}|\lambda|^2\|u_{\sharp}\|_{L^2(0,\omega)} \le C\,(\|\hat{f_1}\|_{L^2(0,\omega)}+\sqrt{\eps}\|\hat{f_2}\|_{L^2(0,\omega)})$. Since $\|d_{\theta}^2u_{\sharp}\|_{L^2(0,\omega)}=|\lambda|^2\|u_{\sharp}\|_{L^2(0,\omega)}$, we infer 
\begin{equation}\label{Estimate1}
\sqrt{\eps}(\|u_{\sharp}\|_{H^2(0,\omega)}+|\lambda|^2\|u_{\sharp}\|_{L^2(0,\omega)}) \le C\,(\|\hat{f_1}\|_{L^2(0,\omega)}+\sqrt{\eps}\|\hat{f_2}\|_{L^2(0,\omega)}).
\end{equation}
Now, let us derive a similar estimate for $\lambda_\sharp$.
From the equation before (\ref{EstimationInterm1}), we have
\begin{equation}\label{EstimationInterm1Phi}
\|\lambda_{\sharp}\|_{L^2(0,\omega)}^2 =\bigg|\cfrac{\eps\cos(\lambda\omega) d_{\theta}u_0(0)-d_{\theta}\lambda_0(\omega)}{4\lambda(1+\eps\,\cos^2(\lambda\omega))}\bigg|^2(\tau^{-1}\sinh(2\tau\omega)-\beta^{-1}\sin(2\beta\omega)).
\end{equation}
We infer
\[
|\lambda|^4\|\lambda_{\sharp}\|_{L^2(0,\omega)}^2 \le C\,|\lambda|\,\left(\cfrac{ \eps^2|\cos(\lambda\omega)d_{\theta}u_0(0)|^2}{|1+\eps\,\cos^2(\lambda\omega)|^2}+\cfrac{ \,|d_{\theta}\lambda_0(\omega)|^2}{|1+\eps\,\cos^2(\lambda\omega)|^2}\right)\,e^{2|\tau|\omega}.
\]
Working as in (\ref{EstimationInterm4}) and (\ref{EstimationInterm4Bis}), we find
\[
|\cos(\lambda\omega)d_{\theta}u_0(0)|^2 \le C\,\|\hat{f_1}\|^2_{L^2(0,\omega)}\,e^{2|\tau|\omega}/|\lambda|
\]
and
\[
|d_{\theta}\lambda_0(\omega)|^2=C\,\eps^2\,\|\hat{f_2}\|^2_{L^2(0,\omega)}e^{2|\tau|\omega}/(|\lambda|\,|\cos(\lambda\omega)|^2) \le C\,\eps^2\,\|\hat{f_2}\|^2_{L^2(0,\omega)}/|\lambda|.
\]
By using again Lemmas \ref{lemmaEstimate} and \ref{lemmaEstimate2} of Appendix B, we deduce that $|\lambda|^2\|\lambda_{\sharp}\|_{L^2(0,\omega)} \leq C\,(\|\hat{f_1}\|_{L^2(0,\omega)}+\sqrt{\eps}\|\hat{f_2}\|_{L^2(0,\omega)})$. Since $\|d_{\theta}^2\lambda_{\sharp}\|_{L^2(0,\omega)}=|\lambda|^2\|\lambda_{\sharp}\|_{L^2(0,\omega)}$, we infer 
\begin{equation}\label{Estimate2}
\|\lambda_{\sharp}\|_{H^2(0,\omega)}+|\lambda|^2\|\lambda_{\sharp}\|_{L^2(0,\omega)} \le C\,(\|\hat{f_1}\|_{L^2(0,\omega)}+\sqrt{\eps}\|\hat{f_2}\|_{L^2(0,\omega)}).
\end{equation}
From the decomposition $(\hat{u}_\eps,\hat{\la}_\eps)=(u_0,\lambda_0)+(u_{\sharp},\lambda_{\sharp})$, using estimates (\ref{EstimateUniform0_part_u}), (\ref{EstimateUniform0_part_phi}), (\ref{Estimate1}) and (\ref{Estimate2}), we finally obtain 
\[
\sqrt{\eps}(\|\hat{u}_\eps\|_{H^2(0,\omega)}+|\lambda|^2\|\hat{u}_\eps\|_{L^2(0,\omega)})+\|\hat{\la}_\eps\|_{H^2(0,\omega)}+|\lambda|^2\|\hat{\la}_\eps\|_{L^2(0,\omega)} \le C\,(\|\hat{f_1}\|_{L^2(0,\omega)}+\sqrt{\eps}\|\hat{f_2}\|_{L^2(0,\omega)}).
\]
It remains to integrate the above estimate on $\ell_{-\beta}$ following the definition of the norm $\|\cdot\|_{\beta,0}$ given by (\ref{norm}).
\end{proof}
\noindent We now consider a problem in the infinite cone $K$ of vertex $S$ and angle $\omega$ which is associated with the problem (\ref{Pr_cauchy_strip}) in the strip via the change of variable $t=\ln r$. 
For $\beta \in \mathbb{R}$, we define the operator $\mathcal{C}_\beta: \mathcal{D}(\mathcal{C}_\beta) \longrightarrow \mathcal{R}(\mathcal{C}_\beta)$ such that $(f_1,f_2)=\mathcal{C}_\beta(u_\eps,\la_\eps)$ with
\be (f_1,f_2)=(-\Delta u_\eps,-\Delta \la_\eps/\eps) \label{problem_C}\ee
\[
\begin{array}{llcl}
\mbox{ and } & \mathcal{D}(\mathcal{C}_\beta)&=&\{(u_\eps,\la_\eps) \in \mathring{V}^1_{\beta-1,0}(K) \cap V_\beta^{2}(K) \times \mathring{V}^1_{\beta-1,\omega}(K) \cap  V_\beta^{2}(K),\\
&&&\ \partial_\nu u_\eps-\partial_\nu \la_\eps=0\,\,{\rm on}\,\, \partial K_0,\quad \eps \partial_\nu u_\eps+ \partial_\nu \la_\eps =0 \,\,{\rm on}\,\, \partial K_\omega\}\\[3pt]
&\mathcal{R}(\mathcal{C}_\beta)&=&V_\beta^{0}(K) \times V_\beta^{0}(K) .
\end{array}
\]
We have the following corollary to Theorem \ref{th_betaqr}.
\begin{corollary}
\label{equiv_B_C}
If $\beta-1 \notin \{(\pi/2+n\pi)/\omega,\,n\in\mathbb{Z}\}$, then the operator $\mathcal{C}_\beta$ is an isomorphism.
Moreover, there exists a constant $C>0$ such that for all
$(f_1,f_2) \in \mathcal{R}(\mathcal{C}_\beta)$, the solution $(u_\eps,\la_\eps) \in \mathcal{D}(\mathcal{C}_\beta)$ to problem (\ref{problem_C})
satisfies
\begin{equation}\label{EstimateUniform_C}
\sqrt{\eps}\|u_\eps\|_{V^2_{\beta}(K)}+\|\la_\eps\|_{V^2_{\beta}(K)} \leq C\,(\|f_1\|_{V^0_{\beta}(K)}+\sqrt{\eps}\|f_2\|_{V^0_{\beta}(K)}).
\end{equation}
\end{corollary}
\begin{proof}
The equation $-\Delta u=f$ in $K$ writes in polar coordinates
\[-((r\partial_r)^2+\partial^2_\theta)u=r^2f,
\]
which by using the operator $\mathcal{E}$ implies that
\[r^2 \mathcal{C}_\beta= \mathcal{E}^{-1} \mathcal{B}_{\beta-1} \mathcal{E}.\]
Indeed the operator $\mathcal{E}$ maps $V_\beta^{2}(K)$ to the space $W^{2}_{\beta-2+1}(B)=W^{2}_{\beta-1}(B)$, the space $\mathring{V}^1_{\beta-1,0}(K)$ to the space $\mathring{W}^1_{\beta-1,0}(B)$ and the space $\mathring{V}^1_{\beta-1,\omega}(K)$ to the space $\mathring{W}^1_{\beta-1,\omega}(B)$, which implies that
$\mathcal{E}$ is an isomorphism from $\mathcal{D}(\mathcal{C}_\beta)$ to $\mathcal{D}(\mathcal{B}_{\beta-1})$.
In addition, the operator $\mathcal{B}_{\beta-1}$ is an isomorphism 
if $\beta-1 \notin \{(\pi/2+n\pi)/\omega,\,n\in\mathbb{Z}\}$.
Lastly, the operator $\mathcal{E}^{-1}$ maps the space $W^{0}_{\beta-1}(B)$ to the space $V^0_{\beta-2}(K)$, which implies that 
$\mathcal{E}^{-1}$ is an isomorphism from $\mathcal{R}(\mathcal{B}_{\beta-1})$ to $\mathcal{R}(\mathcal{C}_{\beta-2})$.
It remains to remark that the operator $f \mapsto r^{-2} f$ maps the space $V^0_{\beta-2}(K)$ to the space $V^0_{\beta}(K)$, and is hence an isomorphism from $\mathcal{R}(\mathcal{C}_{\beta-2})$ to $\mathcal{R}(\mathcal{C}_{\beta})$. This completes the proof of the first part.\\
The estimate relies again on the identity $r^2 \mathcal{C}_\beta= \mathcal{E}^{-1} \mathcal{B}_{\beta-1} \mathcal{E}$, on the fact that
$\mathcal{E}$ is an isomorphism from $V^2_{\beta}(K)$ to $W^2_{\beta-1}(B)$, on the estimate (\ref{EstimateUniform_B}) with $\beta$ replaced by $\beta-1$ and of the fact that
$r^{-2}\mathcal{E}^{-1}$ is an isomorphism from $W^0_{\beta-1}(B)$ to $V^0_\beta(K)$.
\end{proof}
\noindent In order to link the solutions of problem (\ref{problem_C}) obtained for different $\beta$, we need to compute the adjoint of the symbol $\mathscr{L}_\eps$ defined in (\ref{DefSymbol}) and to specify its eigenvalues and eigenfunctions.
\begin{lemma}
The adjoint of the symbol $\mathscr{L}_\eps(\la)$ is the symbol
$\mathscr{L}^*_\eps(\la): \mathcal{D}(\mathscr{L}^*_\eps) \rightarrow \mathcal{R}(\mathscr{L}^*_\eps)$ with
\[\mathscr{L}^*_\eps(\la)(g_\eps,h_\eps)=(-(\overline{\la}^2+d^2_\theta)g_\eps,-\frac{1}{\eps}(\overline{\la}^2+d^2_\theta)h_\eps)\]
\[
\begin{array}{llcl}
&\mathcal{D}(\mathscr{L}^*_\eps)&=&\{(g_\eps,h_\eps) \in H^{2}(0,\omega) \times H^{2}(0,\omega),\,d_\theta g _\eps(\omega)=0,\,d_\theta h_\eps(0)=0,\\[2pt]
&&&\quad g_\eps(\omega)- h_\eps(\omega) =0,\,\eps g_\eps(0)+ h_\eps(0)=0\},\\[3pt]
\mbox{ and } & \mathcal{R}(\mathscr{L}^*_\eps)&=&L^2(0,\omega) \times L^2(0,\omega).
\end{array}
\]
\end{lemma}
\begin{proof}
For $(\varphi,\psi) \in \mathcal{D}(\mathscr{L}_\eps)$ and $(g,h) \in \mathcal{D}(\mathscr{L}^*_\eps)$,
we have by an integration by parts formula
\[\int_0^\omega -(\la^2+ d^2_\theta)\varphi\, \overline{g}\,d\theta+ \int_0^\omega -\frac{1}{\eps}(\la^2+ d^2_\theta)\psi\, \overline{h}\,d\theta\]
\[=\int_0^\omega \varphi \,\, \overline{-(\overline{\la}^2+ d^2_\theta) g}\,d\theta+ \int_0^\omega \psi \,\,\overline{-\frac{1}{\eps}(\overline{\la}^2+ d^2_\theta)h}\,d\theta\]
\[-d_\theta \varphi(\omega) \overline{g}(\omega) + d_\theta \varphi(0) \overline{g}(0) + \varphi(\omega)d_\theta \overline{g}(\omega)-\varphi(0)d_\theta \overline{g}(0)\]
\[-\frac{1}{\eps} d_\theta \psi(\omega) \overline{h}(\omega)+\frac{1}{\eps} d_\theta \psi(0) \overline{h}(0) +\frac{1}{\eps} \psi(\omega)d_\theta \overline{h}(\omega)- \frac{1}{\eps} \psi(0)d_\theta \overline{h}(0).\]
It is readily seen that all the boundary terms vanish due to the boundary conditions satisfied by $(\varphi,\psi)$ and $(g,h)$ at $\theta=0$ and $\theta=\omega$.
This completes the proof.
\end{proof}
\begin{lemma}
\label{eigenvalue_qr_adjoint}
The eigenvalues of the symbol $\mathscr{L}^*_\eps$ are the same as that of $\mathscr{L}_\eps$ and are given by
\[\la^\pm_n=\frac{1}{\omega}\left(\frac{\pi}{2}+n\pi \pm i \ln \gamma_\eps\right),\quad n \in \mathbb{Z},\]
with 
\[\gamma_\eps=\sqrt{1+\frac{1}{\eps}}+\sqrt{\frac{1}{\eps}}.\]
The corresponding (non normalized) eigenfunctions are given by
\[g^\pm_n(\theta)=\cos(\overline{\la^\pm_n}\omega)\cos(\overline{\la^\pm_n} (\theta-\omega)),\qquad\qquad h_n^\pm(\theta)=\cos(\overline{\la^\pm_n} \theta).\]
\end{lemma}
\noindent The proof of Lemma \ref{eigenvalue_qr_adjoint} is the same as the proof of Lemma \ref{eigenvalue_qr} and is therefore not given.
Lastly, we will need a biorthogonality relationship between the eigenfunctions of $\mathscr{L}_\eps$ and that of $\mathscr{L}^*_\eps$.
\begin{lemma}
\label{biorthogonal}
Assume that $j,k \in \mathbb{Z}$ and $\nu,\mu =\pm$ satisfy either $j+k \neq -1$ or $\mu+\nu \neq 0$.
The eigenfunctions $(\varphi^\pm_n,\psi^\pm_n)$ of $\mathscr{L}_\eps$ and the eigenfunctions $(g^\pm_n,h^\pm_n)$ of $\mathscr{L}^*_\eps$
satisfy
\[\int_0^\omega \varphi^\nu_k \overline{g^\mu_j}\,d\theta + \frac{1}{\eps}\int_0^\omega \psi^\nu_k \overline{h^\mu_j} d\theta= \delta_{\nu\mu}\delta_{kj} d_k, \]
with 
\be d_k=(-1)^{k+1}\frac{\omega}{\eps} \sqrt{1+\frac{1}{\eps}}.\label{dk}\ee
\end{lemma}
\begin{proof}
On the one hand, the assumption $j+k \neq -1$ or $\mu+\nu \neq 0$ is equivalent to $\la_j^\mu \neq -\la_k^\nu$.
Let us first assume that $k\neq j$ and $\nu=\mu=+$, which implies on the other hand that $\la_j^\mu \neq \la_k^\nu$.
Skipping the sign $+$,
we have
\[-\la_k^2 \int_0^\omega \left(\varphi_k \overline{g_j}+ \frac{1}{\eps} \psi_k \overline{h_j}\right)\,d\theta=
\int_0^\omega \left(\Delta \varphi_k \overline{g_j} + \frac{1}{\eps} \Delta \psi_k \overline{h_j}\right)\,d\theta\]
\[=\int_0^\omega \left(\varphi_k \overline{\Delta g_j} + \frac{1}{\eps} \psi_k \overline{\Delta h_j}\right)\,d\theta=-\la_j^2 \int_0^\omega \left(\varphi_k \overline{g_j}+ \frac{1}{\eps} \psi_k \overline{h_j}\right)\, d\theta.\]
Since $\la_j^2 \neq \la_k^2$, this implies that for $j\neq k$ and $\nu=\mu=+$, we have 
\[\int_0^\omega \left(\varphi^\nu_k \overline{g^\mu_j}+ \frac{1}{\eps} \psi^\nu_k \overline{h^\mu_j}\right)\,d\theta=0.\]
We clearly obtain the same result each time that $(k,\nu) \neq (j,\mu)$.
Let us now assume that $k=j$ and $\nu=\mu$.
We have
\[\int_0^\omega \varphi_k^\nu \overline{g^\nu_k}\,d\theta=\cos^2(\la_k^\nu \omega)\int_0^\omega \sin(\la_k^\nu \theta)\cos(\la_k^\nu(\theta-\omega))\,d\theta \]
and
\[\int_0^\omega \psi_k^\nu \overline{h_k^\nu}\,d\theta=\int_0^\omega \sin(\la^\nu_k(\theta-\omega))\cos(\la^\nu_k \theta)\,d\theta=-\int_0^\omega \sin(\la_k^\nu \theta)\cos(\la_k^\nu(\theta-\omega))\,d\theta.\]
Given that $1+\eps\cos^2(\la_k^\nu \omega)=0$, we obtain
\be \int_0^\omega \varphi^\nu_k \overline{g^\nu_k}\,d\theta + \frac{1}{\eps}\int_0^\omega \psi^\nu_k \overline{h^\nu_k} d\theta=-\frac{2}{\eps}
\int_0^\omega \sin(\la_k^\nu \theta)\cos(\la_k^\nu(\theta-\omega))\,d\theta.\label{calcul1}\ee
But a direct calculus gives 
\begin{equation}\label{calcul2}
\begin{array}{lcl}
\dsp\int_0^\omega \sin(\la_k^\nu \theta)\cos(\la_k^\nu(\theta-\omega))\,d\theta&=&\dsp\frac{1}{2}\int_0^\omega \sin(\la_k^\nu(2\theta-\omega))\,d\theta + \frac{1}{2}\int_0^\omega \sin(\la_k^\nu \omega)\,d\theta \\[8pt]
&=&\dsp\frac{1}{2}\,\omega \sin(\la_k^\nu\omega)=\frac{1}{2}\,\omega \sin \left(\frac{\pi}{2}+k\pi +i\nu \ln \gamma_\eps\right)\\[8pt]
&=&\dsp \frac{(-1)^k}{2}\,\omega \cos(i\nu\ln \gamma_\eps)
=\frac{(-1)^k}{2}\,\omega\cosh(\ln \gamma_\eps).
\end{array}
\end{equation}
Since $\gamma_\eps=\sqrt{1+1/\eps}+\sqrt{1/\eps}$, we find
\be \cosh(\ln \gamma_\eps)=\frac{1}{2}\left(\gamma_\eps+\frac{1}{\gamma_\eps}\right)=\sqrt{1+\frac{1}{\eps}}\label{calcul3}.\ee
Using (\ref{calcul3}) and (\ref{calcul2}) in (\ref{calcul1}), we get the desired result.
\end{proof}
\noindent In the next theorem, we compare two solutions of problem (\ref{Pr_cauchy_strip}) associated with two different values of $\beta$.
\begin{theorem}
\label{beta1_beta2}
Assume that $\beta_1<\beta_2$ are two real numbers such that $\beta_j \notin \{(\pi/2+n\pi)/\omega,\,n\in\mathbb{Z}\}$, $j=1,2$.
Let us denote by $\la^\nu_1,\la^\nu_2,\dots,\la^\nu_N$, with $\nu=\pm$, the eigenvalues of $\mathscr{L}_\eps$ in the strip $-\beta_2<{\rm Re}\,\la< -\beta_1$.
For $(f_1,f_2) \in \mathcal{R}(\mathcal{B}_{\beta_1}) \cap \mathcal{R}(\mathcal{B}_{\beta_2})$, the solutions $(u_{\beta_1},\la_{\beta_1}) \in \mathcal{D}(\mathcal{B}_{\beta_1})$ and 
$(u_{\beta_2},\la_{\beta_2}) \in \mathcal{D}(\mathcal{B}_{\beta_2})$ to the problems $\mathcal{B}_{\beta_1} (u_{\beta_1},\la_{\beta_1})=(f_1,f_2)$
and $\mathcal{B}_{\beta_2} (u_{\beta_2},\la_{\beta_2})=(f_1,f_2)$ satisfy the relationship
\be (u_{\beta_2},\la_{\beta_2})=(u_{\beta_1},\la_{\beta_1})+\sum_{\nu \in\{\pm\}} \sum_{k=1}^N c^\nu_k e^{\la^\nu_k t} (\varphi^\nu_k,\psi^\nu_k),\label{diff_u1_u2}\ee
where $(\varphi^\nu_k,\psi^\nu_k)$ is the eigenvector of $\mathscr{L}_\eps$ associated with the eigenvalue $\la^\nu_k$ (see Lemma \ref{eigenvalue_qr}) and 
\begin{equation}\label{FormulaCoeff}
c^\nu_k=\frac{1}{2\la_k^\nu d_k}\Big((f_1,e^{-\overline{\la_k^\nu} t}g^\nu_k)_{L^2(B)}+(f_2,e^{-\overline{\la_k^\nu} t}h^\nu_k)_{L^2(B)}\Big).
\end{equation}
Here $(g^\nu_k,h^\nu_k)$ stand for the eigenvector of $\mathscr{L}^*_\eps$ associated with the eigenvalue $\la^\nu_k$ (see Lemma \ref{eigenvalue_qr_adjoint}) and $d_k$ is given by (\ref{dk}).
\end{theorem}
\begin{proof}
The first part of the theorem is obtained 
by using the residue theorem as in the proof of \cite[Theorem 5.1.1]{kozlov_mazya_rossmann}. Now we establish (\ref{FormulaCoeff}). Let us introduce a cut-off function $\xi \in \mathbb{R}$ such that $\xi(t)=0$ for $t \leq t_1$ and $\xi(t)=1$ for $t \geq t_2$, with $t_1<t_2$. From (\ref{diff_u1_u2}) and using the short notation $\mathcal{B}=(-\Delta,-\Delta/\eps)$, we have
\[-\mathcal{B}(\xi(u_{\beta_2}-u_{\beta_1}),\xi(\la_{\beta_2}-\la_{\beta_1}))=\sum_{\nu \in\{\pm\}}\sum_{k=1}^N c^\nu_k \Big(\Delta (\xi(e^{\la^\nu_k t} \varphi^\nu_k))+ \frac{1}{\eps}\Delta(\xi(e^{\la^\nu_k t} \psi^\nu_k))\Big).\]
We observe that $\Delta (\xi(e^{\la^\nu_k t} \varphi^\nu_k))$ and $\Delta (\xi(e^{\la^\nu_k t} \varphi^\nu_k))$ are non vanishing only on $[t_1,t_2] \times [0,\omega]$, which implies that for $j=1,\dots,N$ and $\mu=\pm$,
\[-\Big(\mathcal{B}(\xi(u_{\beta_2}-u_{\beta_1}),\xi(\la_{\beta_2}-\la_{\beta_1})),(e^{-\overline{\la^\mu_j}t}g_j,e^{-\overline{\la^\mu_j}t}h_j)\Big)_{L^2(B) \times L^2(B)}\]
\[
=\sum_{\nu\in\{\pm\}}\sum_{k=1}^N c^\nu_k \Big(\Delta (\xi(e^{\la^\nu_k t} \varphi^\nu_k)),e^{-\overline{\la^\mu_j}t}g_j\Big)_{L^2((t_1,t_2) \times (0,\omega))} \]
\[+\frac{1}{\eps}\sum_{\nu\in\{\pm\}}\sum_{k=1}^N c^\nu_k
\Big(\Delta (\xi(e^{\la^\nu_k t} \psi^\nu_k)),e^{-\overline{\la^\mu_j}t}h_j\Big)_{L^2((t_1,t_2) \times (0,\omega))} .\]
By an integration by parts formula in the domain $(t_1,t_2) \times (0,\omega)$ and by using that $\Delta (e^{-\overline{\la^\mu_j} t}g_j)=0$ and
$\Delta (e^{-\overline{\la^\mu_j} t}h_j)=0$,
we get that
\[-\Big(\mathcal{B}(\xi(u_{\beta_2}-u_{\beta_1}),\xi(\la_{\beta_2}-\la_{\beta_1})),(e^{-\overline{\la^\mu_j}t}g_j,e^{-\overline{\la^\mu_j}t}h_j)\Big)_{L^2(B) \times L^2(B)}\]
\[
=\sum_{\nu \in\{\pm\}}\sum_{k=1}^N c^\nu_k \Big(\la^\nu_k e^{\la^\nu_k t_2} \varphi^\nu_k,e^{-\overline{\la^\mu_j} t_2}g^\mu_j\Big)_{L^2(0,\omega)}
+ \frac{1}{\eps} \sum_{\nu \in\{\pm\}}\sum_{k=1}^N c^\nu_k \Big((\la^\nu_k e^{\la^\nu_k t_2} \psi^\mu_k,e^{-\overline{\la^\mu_j} t_2}h^\mu_j\Big)_{L^2(0,\omega)}\]
\[
-\sum_{\nu \in\{\pm\}}\sum_{k=1}^N c^\nu_k \Big(e^{\la^\nu_k t_2} \varphi^\nu_k,-\overline{\la^\mu_j} e^{-\overline{\la^\mu_j} t_2}g^\mu_j\Big)_{L^2(0,\omega)}
-\frac{1}{\eps} \sum_{\nu \in\{\pm\}}\sum_{k=1}^N c^\nu_k \Big(e^{\la^\nu_k t_2} \psi^\nu_k,-\overline{\la^\mu_j} e^{-\overline{\la^\mu_j} t_2}h^\mu_j\Big)_{L^2(0,\omega)}.
\]
In view of the biorthogonality relationships of Lemma \ref{biorthogonal} and due to the fact that in case $\la^\nu_k=-\la^\mu_j$ (that is $j+k=-1$ and $\nu+\mu=0$) the first and third terms within the brackets above compensate one another as well as the second and fourth terms, 
we end up with
\be -\Big(\mathcal{B}(\xi(u_{\beta_2}-u_{\beta_1}),\xi(\la_{\beta_2}-\la_{\beta_1})),(e^{-\overline{\la^\mu_j}t}g_j,e^{-\overline{\la^\mu_j}t}h_j)\Big)_{L^2(B) \times L^2(B)}=2 \la^\mu_j c_j^\mu d_j.\label{step1}\ee
On the other end, since $\beta_1 < \beta_2$, the function $u_{\beta_2}$ is more decreasing than $u_{\beta_1}$ at $+\infty$. 
And the situation is inverted at $-\infty$. The same property holds for $\la_{\beta_2}$ and $\la_{\beta_1}$.
Since $\xi$ vanishes at $-\infty$, we have that $(\xi u_{\beta_2}, \xi \la_{\beta_2}) \in \mathcal{D}(\mathcal{B}_{\beta_1}) \cap \mathcal{D}(\mathcal{B}_{\beta_2})$.
Using an integration by parts in $B$ and the fact that $-\beta_2 < {\rm Re}\,\la_j$, we obtain that
\be \Big(\mathcal{B}(\xi u_{\beta_2},\xi \la_{\beta_2}),(e^{-\overline{\la^\mu_j}t}g_j,e^{-\overline{\la^\mu_j}t}h_j)\Big)_{L^2(B) \times L^2(B)}
=0.\label{step2}\ee
With the same argument, we obtain
\be \Big(\mathcal{B}((1-\xi) u_{\beta_1},(1-\xi)\la_{\beta_1}),(e^{-\overline{\la^\mu_j}t}g_j,e^{-\overline{\la^\mu_j}t}h_j)\Big)_{L^2(B) \times L^2(B)}
=0.\label{step3}\ee
By combining (\ref{step1}), (\ref{step2}) and (\ref{step3}), we get
\[
\begin{array}{ll}
& 2 \la^\mu_j c_j^\mu d_j \\[4pt]
=& \Big(\mathcal{B}(\xi(u_{\beta_1}-u_{\beta_2}),\xi(\la_{\beta_1}-\la_{\beta_2})),(e^{-\overline{\la^\mu_j}t}g_j,e^{-\overline{\la^\mu_j}t}h_j)\Big)_{L^2(B) \times L^2(B)}\\[6pt]
=& \Big(\mathcal{B}(\xi u_{\beta_1},\xi \la_{\beta_1}),(e^{-\overline{\la^\mu_j}t}g_j,e^{-\overline{\la^\mu_j}t}h_j)\Big)_{L^2(B) \times L^2(B)}\\[6pt]
=& \Big(\mathcal{B}(u_{\beta_1},\la_{\beta_1}),(e^{-\overline{\la^\mu_j}t}g_j,e^{-\overline{\la^\mu_j}t}h_j)\Big)_{L^2(B) \times L^2(B)}\\[6pt]
 & \qquad-\Big(\mathcal{B}((1-\xi)u_{\beta_1},(1-\xi)\la_{\beta_1}),(e^{-\overline{\la^\mu_j}t}g_j,e^{-\overline{\la^\mu_j}t}h_j)\Big)_{L^2(B) \times L^2(B)}\\[6pt]
 =&\Big(\mathcal{B}(u_{\beta_1},\la_{\beta_1}),(e^{-\overline{\la^\mu_j}t}g_j,e^{-\overline{\la^\mu_j}t}h_j)\Big)_{L^2(B) \times L^2(B)}=\Big((f_1,f_2),(e^{-\overline{\la^\mu_j}t}g_j,e^{-\overline{\la^\mu_j}t}h_j)\Big)_{L^2(B) \times L^2(B)},
\end{array}
\]
which completes the proof.
\end{proof}
\noindent From the previous theorem in the strip, we obtain the following corollary in the infinite cone by using the identity $r^2 \mathcal{C}_\beta= \mathcal{E}^{-1} \mathcal{B}_{\beta-1} \mathcal{E}$ 
\begin{corollary}
\label{beta1_beta2_C}
Assume that $\beta_1<\beta_2$ are two real numbers such that $\beta_j-1 \notin \{(\pi/2+n\pi)/\omega,\,n\in\mathbb{Z}\}$, $j=1,2$.
Let us denote by $\la^\nu_1,\la^\nu_2,\dots,\la^\nu_N$, with $\nu=\pm$, the eigenvalues of $\mathscr{L}_\eps$ in the strip $-\beta_2+1<{\rm Re}\,\la< -\beta_1+1$.
For $(f_1,f_2) \in \mathcal{R}(\mathcal{C}_{\beta_1}) \cap \mathcal{R}(\mathcal{C}_{\beta_2})$, the solutions $(u_{\beta_1},\la_{\beta_1}) \in \mathcal{D}(\mathcal{C}_{\beta_1})$ and 
$(u_{\beta_2},\la_{\beta_2}) \in \mathcal{D}(\mathcal{C}_{\beta_2})$ to the problems $\mathcal{C}_{\beta_1} (u_{\beta_1},\la_{\beta_1})=(f_1,f_2)$
and $\mathcal{C}_{\beta_2} (u_{\beta_2},\la_{\beta_2})=(f_1,f_2)$ satisfy the relationship
\[(u_{\beta_2},\la_{\beta_2})=(u_{\beta_1},\la_{\beta_1})+\sum_{\nu \in\{\pm\}} \sum_{k=1}^N c^\nu_k r^{\la^\nu_k} (\varphi^\nu_k,\psi^\nu_k),\]
where
\[c^\nu_k=\frac{1}{2\la_k^\nu d_k}\Big((f_1,r^{-\overline{\la_k^\nu}}g^\nu_k)_{L^2(K)}+(f_2,r^{-\overline{\la_k^\nu}}h^\nu_k)_{L^2(K)}\Big).\]
\end{corollary}
\begin{remark}
For real valued functions $f_1$ and $f_2$, we have $\overline{c^\nu_k}=c^{-\nu}_k$ for all $k \in\{1,\dots,N\}$.
\end{remark}
\noindent We end up with the main proposition of this section.
\begin{proposition}
\label{mixed}
Assume that $S$ is the vertex of a corner of mixed type.
Let us consider $s <1+\pi/(2\omega)$ if $\omega \geq \pi/2$ and $s=2$ otherwise. 
For $f \in L^2(\Omega)$ and $\eps>0$, the solution $(u_\eps,\la_\eps) \in V_{0} \times \tilde{V}_0$ to the problem (\ref{Pr_cauchy_simple}) is such that $\zeta u_\eps$ and $\zeta \la_\eps$ belong to $H^{s}(\Omega)$ and there exists a constant $C>0$ which depends only on the geometry such that
\[\forall \eps \in (0,1], \quad \eps\|\zeta u_\eps\|_{H^{s}(\Omega)} + \sqrt{\eps}\|\zeta \la_\eps\|_{H^{s}(\Omega)} \leq C \|f\|_{L^2(\Omega)}.\]
If in addition we assume that $f$ is such that problem (\ref{cauchy_strong_simple}) has a (unique) solution $u$, then
\[\forall \eps \in (0,1], \quad \sqrt{\eps}\|\zeta u_\eps\|_{H^{s}(\Omega)} + \|\zeta \la_\eps\|_{H^{s}(\Omega)} \leq C \|u\|_{H^1(\Delta,\Omega)}.\]
\end{proposition}
\begin{proof}
The pair $(v_\eps,\mu_\eps),=(\zeta u_\eps,\zeta \la_\eps)$, where $(u_\eps,\la_\eps) \in V_0 \times \tilde{V}_0$ solves (\ref{Pr_cauchy_simple}),  
satisfies the problem
\be \left\{
\begin{array}{rcll}
\displaystyle  -\Delta v_\eps  &=& g_\eps&  \mbox{in } K\\
\displaystyle -\Delta \mu_\eps  &=& \eps\,h_\eps &  \mbox{in } K\\
 v_\eps&=&0  & \mbox{on }\partial K_0 \\
 \partial_\nu v_\eps -\partial_\nu \mu_\eps&=&0 & \mbox{on } \partial K_0 \\
 \mu_\eps&=&0  & \text{on }\partial K_\omega \\
 \eps\, \partial_\nu v_\eps + \partial_\nu \mu_\eps&=&0 & \mbox{on }\partial K_\omega,
\end{array}
\right.
\label{cauchy_cone}
\ee
with $\dsp(g_\eps,h_\eps)=\Big(-\Delta \zeta u_\eps -2\nabla \zeta\cdot \nabla u_\eps +\zeta \frac{f}{1+\eps},-\Delta \zeta \frac{\la_\eps}{\eps} -2\nabla \zeta\cdot \nabla\Big( \frac{\la_\eps}{\eps}\Big) -\zeta \frac{f}{1+\eps}\Big)$.\\
\newline
Let us study the regularity of $v_\eps$, $\mu_\eps$ by using the properties of the operator $\mathcal{C}_\beta$ defined in (\ref{problem_C}). To proceed, in particular, we will exploit the results of Corollaries \ref{equiv_B_C} and \ref{beta1_beta2_C}.\\
\newline
First, observing that $V_0^0(K)=L^2(K)$ and that $\zeta$ is compactly supported, we deduce that $(g_\eps,h_\eps) \in V_0^0(K) \times V_0^0(K)$. And more generally, we have
$(g_\eps,h_\eps) \in \mathcal{R}(\mathcal{C}_\beta)=V_\beta^0(K) \times V_\beta^0(K)$ for all $\beta \geq 0$. For $\beta=1$, there holds $\beta-1\notin \{(\pi/2+n\pi)/\omega,\,n \in \mathbb{Z}\}$. From Corollary \ref{equiv_B_C}, we infer that $\mathcal{C}_1$ is an isomorphism from  
\[
\begin{array}{ll}
\mathcal{D}(\mathcal{C}_1)=&\{(v_\eps,\mu_\eps) \in \mathring{V}^1_{0,0}(K) \cap V_1^{2}(K) \times \mathring{V}^1_{0,\omega}(K) \cap  V_1^{2}(K),\\[3pt]
&\quad\partial_\nu v_\eps-\partial_\nu \mu_\eps=0\,\,{\rm on}\,\, \partial K_0,\quad \eps \partial_\nu v_\eps+ \partial_\nu \mu_\eps =0 \,\,{\rm on}\,\, \partial K_\omega\}
\end{array}
\]
to $\mathcal{R}(\mathcal{C}_1)\ni(g_\eps,h_\eps)$. Let us denote by $(v_\eps^1,\mu_\eps^1) \in \mathcal{D}(\mathcal{C}_1)$ the unique element of $\mathcal{D}(\mathcal{C}_1)$ such that $\mathcal{C}_1 (v_\eps^1,\mu_\eps^1)=(g_\eps,h_\eps)$. Corollary \ref{equiv_B_C} ensures that there is a constant $C$ such that
\[
\sqrt{\eps}\|v_\eps^1\|_{V^2_1(K)}+\|\mu_\eps^1\|_{V^2_1(K)} \leq C\,(\|g_\eps\|_{V^0_1(K)}+\sqrt{\eps}\|h_\eps\|_{V^0_1(K)}).
\]
Let us prove that $(v_\eps^1,\mu_\eps^1)$ coincides with $(v_\eps,\mu_\eps)$.
We have $(v_\eps,\mu_\eps) \in \mathring{V}_{0,0}^1 \times \mathring{V}_{0,\omega}^1$.
Indeed, $v_\eps$ vanishes for $r \geq b$ and from Poincar\'e's inequality, there holds
\[\int_{K_b} \frac{1}{r^2}v^2_\eps\,dx \leq C \int_{K_b} |\nabla v_\eps|^2\,dx.\]
The same inequality is valid for $\mu_\eps$.
Next, let us introduce $\upsilon \in \mathscr{C}^{\infty}_0(\overline{K})$ such that $\upsilon$ vanishes in a vicinity of $\partial K_0$ and $\psi \in \mathscr{C}^{\infty}_0(\overline{K})$ such that $\psi$ vanishes in a vicinity of $\partial K_\omega$.
It is easy to check that $(v_\eps,\mu_\eps)$ solves 
\begin{equation}\label{PbVaria_K}
\left\{\begin{array}{rcl}
\displaystyle \eps\int_K \nabla v_\eps\cdot \nabla \upsilon\,dx + \int_K \nabla \upsilon\cdot \nabla \mu_\eps\,dx&=&\dsp\int_K\eps(g_\eps+h_\eps)\,\upsilon\,dx\\[8pt]
\displaystyle \int_K \nabla v_\eps \cdot \nabla \psi\,dx -\int_K \nabla \mu_\eps\cdot \nabla \psi\,dx &=&\dsp\int_K (g_\eps-\eps\,h_\eps)\,\psi\,dx.
\end{array}
\right.
\end{equation}
One can also verify that $(v_\eps^1,\mu_\eps^1)$ satisfies (\ref{PbVaria_K}). Since $(v_\eps-v_\eps^1,\mu_\eps-\mu_\eps^1) \in \mathring{V}_{0,0}^1 \times \mathring{V}_{0,\omega}^1$, using the density of the set of functions $\upsilon$ (resp. $\psi$) in $\mathring{V}_{0,0}^1$ (resp. in $\mathring{V}_{0,\omega}^1$), we conclude that $(v_\eps,\mu_\eps)=(v_\eps^1,\mu_\eps^1)$. Now we must separate the rest of the analysis according to the configuration.\\
\newline
$\star\,$ Let us first assume that $\omega<\pi/2$. In this case, for $\beta=0$, we have $\beta-1 \notin \{(\pi/2+n\pi)/\omega,\,n \in \mathbb{Z}\}$. Then Corollary \ref{equiv_B_C} guarantees that $\mathcal{C}_0$ is an isomorphism from  
\[
\begin{array}{ll}
\mathcal{D}(\mathcal{C}_0)=&
\{(v_\eps,\mu_\eps) \in \mathring{V}^1_{-1,0}(K) \cap V_0^{2}(K) \times \mathring{V}^1_{-1,\omega}(K) \cap  V_0^{2}(K),\\[3pt]
&\quad\partial_\nu v_\eps-\partial_\nu \mu_\eps=0\,\,{\rm on}\,\, \partial K_0,\quad \eps \partial_\nu v_\eps+ \partial_\nu \mu_\eps =0 \,\,{\rm on}\,\, \partial K_\omega\}
\end{array}
\]
to $\mathcal{R}(\mathcal{C}_0)\ni(g_\eps,h_\eps)$. Let us denote by $(v_\eps^0,\mu_\eps^0) \in \mathcal{D}(\mathcal{C}_0)$ the unique element of $\mathcal{D}(\mathcal{C}_0)$ such that $\mathcal{C}_0 (v_\eps^0,\mu_\eps^0)=(g_\eps,h_\eps)$. Corollary \ref{equiv_B_C} ensures that there is a constant $C$ such that
\be \label{estim_0}
\sqrt{\eps}\|v_\eps^0\|_{V^2_0(K)}+\|\mu_\eps^0\|_{V^2_0(K)} \leq C\,(\|g_\eps\|_{L^2(K)}+\sqrt{\eps}\|h_\eps\|_{L^2(K)})  
\ee
But from Lemma \ref{eigenvalue_qr}, the eigenvalues $\la_n^\pm$ of $\mathscr{L}_\eps$ satisfy ${\rm Re}\,\la^\pm_n=(\pi/2+n \pi)/\omega$, $n \in \mathbb{Z}$. As a consequence when $\omega<\pi/2$, none of them lies in the strip $0<{\rm Re}\,\la<1$. This implies that 
\[(v_\eps,\mu_\eps)=(v_\eps^1,\mu_\eps^1)=(v_\eps^0,\mu_\eps^0).\]
Besides, we observe that $V_0^2(K) \subset H^2(K_b)$. Hence $(v_\eps,\mu_\eps) \in H^2(K_b) \times H^2(K_b)$. Using the estimates 
\be 
\left\{\begin{array}{lcl}
\displaystyle \|g_\eps\|_{L^2(K)} &\leq &C (\|u_\eps\|_{H^1(K_b)}+ \|f\|_{L^2(K_b)})\\[4pt]
\displaystyle \|h_\eps\|_{L^2(K)} &\leq &C (\|\la_\eps\|_{H^1(K_b)}/\eps+ \|f\|_{L^2(K_b)}),
\end{array} \right.
\label{estim_gh}
\ee
then from (\ref{estim_0}), we can write
\[
\begin{array}{ll}
&\sqrt{\eps} \|\zeta u_\eps\|_{H^2(\Omega)} + \|\zeta \la_\eps\|_{H^2(\Omega)} = \sqrt{\eps} \|v_\eps\|_{H^2(K_b)} + \|\mu_\eps\|_{H^2(K_b)}\\[5pt]
\leq &C(\sqrt{\eps} \|v_\eps^0\|_{V_0^2(K)} + \|\mu_\eps^0\|_{V_0^2(K)}) \leq C(\|g_\eps\|_{L^2(K_b)}+\sqrt{\eps}\|h_\eps\|_{L^2(K_b)})\\[5pt]
 \leq & \dsp C(\|u_\eps\|_{H^1(K_b)} + \frac{1}{\sqrt{\eps}}\|\la_\eps\|_{H^1(K_b)} + \|f\|_{L^2(K_b)}).
\end{array}
\]
By using (\ref{estim1}), finally we get 
\[\eps \|\zeta u_\eps\|_{H^2(\Omega)} + \sqrt{\eps} \|\zeta \la_\eps\|_{H^2(\Omega)} \leq C \|f\|_{L^2(\Omega)}.\]
$\star$ Now let us assume that $\omega \geq \pi/2$. Choose $\beta$ such that $0 \leq 1-\pi/(2\omega)<\beta<1$. In this case, since $\beta-1 \notin \{(\pi/2+n\pi)/\omega,\,n \in \mathbb{Z}\}$, the operator $\mathcal{C}_\beta$ is an isomorphism from 
\[
\begin{array}{ll}
\mathcal{D}(\mathcal{C}_\beta)=&\{(v_\eps,\mu_\eps) \in \mathring{V}^1_{\beta-1,0}(K) \cap V_\beta^{2}(K) \times \mathring{V}^1_{\beta-1,\omega}(K) \cap  V_\beta^{2}(K),\\[3pt]
&\quad\partial_\nu u_\eps-\partial_\nu \la_\eps=0\,\,{\rm on}\,\, \partial K_0,\quad \eps \partial_\nu u_\eps+ \partial_\nu \la_\eps =0 \,\,{\rm on}\,\, \partial K_\omega\}
\end{array}
\]
to $\mathcal{R}(\mathcal{C}_\beta)\ni(g_\eps,h_\eps)$. Let us denote by $(v_\eps^\beta,\mu_\eps^\beta) \in \mathcal{D}(\mathcal{C}_\beta)$ the unique element of $\mathcal{D}(\mathcal{C}_\beta)$ such that $\mathcal{C}_\beta (v_\eps^\beta,\mu_\eps^\beta)=(g_\eps,h_\eps)$. Using Corollary \ref{equiv_B_C} and the fact that $g_\eps$, $h_\eps$ are compactly supported with $\beta>0$, we can write 
\be\label{estim_beta} 
\begin{array}{lcl}
\sqrt{\eps}\|v_\eps^\beta\|_{V^2_\beta(K)}+\|\mu_\eps^\beta\|_{V^2_\beta(K)} &\leq &C\,(\|g_\eps\|_{V^0_\beta(K)}+\sqrt{\eps}\|h_\eps\|_{V^0_\beta(K)})\\[2pt]
&\leq &C\,(\|g_\eps\|_{L^2(K)}+\sqrt{\eps}\|h_\eps\|_{L^2(K)}). 
\end{array}
\ee
Again, none of the $\la^\pm_n$ lies in the strip $0<{\rm Re}\,\la<1-\beta$, which implies
that 
\[(v_\eps,\mu_\eps)=(v_\eps^1,\mu_\eps^1)=(v_\eps^\beta,\mu_\eps^\beta).\]
Besides, note that $v_\eps$ and $\mu_\eps$ are supported in $K_b$.
The Theorem 5.2 of \cite{nicaise} ensures that the space $V_\beta^2(K_b)$ is continuously embedded in the interpolate space $[V_0^2(K_b),V_1^2(K_b)]_\theta$ for all
$\theta \in (\beta,1)$.
Since the spaces $V_0^2(K_b)$ and $V_0^1(K_b)$ are continuously embedded in $H^2(K_b)$ and $H^1(K_b)$, respectively, we infer that
$[V_0^2(K_b),V_1^2(K_b)]_\theta$ is continuously embedded in $[H^2(K_b),H^1(K_b)]_\theta=H^{2-\theta}(K_b)$ for all $\theta \in (\beta,1)$.
Since $\beta$ is arbitrarily close to $1-\pi/(2\omega)$, we conclude that 
the space 
$V_\beta^2(K_b)$ is continuously embedded in $H^s(K_b)$ for all $s< 1+\pi/(2\omega)$.
Gathering the estimates (\ref{estim_gh}) and (\ref{estim_beta}), we infer that for all $s< 1 + \pi/(2\omega)$,
\[
\begin{array}{ll}
&\sqrt{\eps} \|\zeta u_\eps\|_{H^s(\Omega)} + \|\zeta \la_\eps\|_{H^s(\Omega)} = \sqrt{\eps} \|v_\eps\|_{H^s(K_b)} + \|\mu_\eps\|_{H^s(K_b)}\\[5pt]
\leq &C(\sqrt{\eps} \|v^\beta_\eps\|_{V_\beta^2(K)} + \|\mu^\beta_\eps\|_{V_\beta^2(K)}) \leq C(\|g_\eps\|_{L^2(K_b)}+\sqrt{\eps}\|h_\eps\|_{L^2(K_b)})\\[5pt]
 \leq & \dsp C(\|u_\eps\|_{H^1(K_b)} + \frac{1}{\sqrt{\eps}}\|\la_\eps\|_{H^1(K_b)} + \|f\|_{L^2(K_b)}).
 \end{array}
 \]
By using the estimate (\ref{estim1}), finally we get $\eps \|\zeta u_\eps\|_{H^s(\Omega)} + \sqrt{\eps} \|\zeta \la_\eps\|_{H^s(\Omega)} \leq C \|f\|_{L^2(\Omega)}$.
\end{proof}
\begin{remark}
By using Corollary \ref{beta1_beta2_C} for $\beta_1=0$ and $\beta_2=1$, we obtain all the singular functions at a corner of mixed type, which are the functions
$r^{\la^\nu_k} (\varphi^\nu_k,\psi^\nu_k)$ which belong to $H^1(K_b)$ but not to $H^2(K_b)$.
The singular functions are readily determined by the value of
${\rm Re}(\la^\nu_k)=(\pi/2+k\pi)/\omega$ for $k \in \mathbb{Z}$: 
\begin{itemize}
\item there is no singularity for $\omega  \leq \pi/2$,
\item singularities are obtained for $k=0$ and $\nu=\pm$ for $\pi/2<\omega \leq 3\pi/2$,
\item singularities are obtained for $k=0$, $k=1$ and $\nu=\pm$ for $\omega>3\pi/2$.
\end{itemize}
Note that this conclusion is very similar to the case of the Laplace equation with mixed Dirichlet-Neumann boundary conditions (see \cite{grisvard_bleu}). 
\end{remark}

\section{Application to error estimates}
\label{application}
In this last section, we use the regularity estimates for solutions of quasi-reversibility problem (\ref{Pr_cauchy_simple}), in particular Theorem \ref{main}, to derive error estimates between the exact solution and the quasi-reversibility solution obtained in the presence of noisy data and with the help of a Finite Element Method.
Let us assume that $\Omega$ is a polygonal domain in two dimensions and
that $u \in H^1(\Omega)$ is the exact solution of problem (\ref{cauchy_strong_simple}) associated with the exact data $f \in L^2(\Omega)$.
In the context of inverse problems, usually $f$ is not available. Only an approximate data $f^\delta \in L^2(\Omega)$ is available, with 
\be \|f^\delta-f\|_{L^2(\Omega)} \leq  \delta,\label{noise}\ee 
where $\delta$ can be viewed as the amplitude of noise.
A natural idea is to solve problem (\ref{Pr_cauchy_simple}) with $f^\delta$ instead of $f$, and a practical way of proceeding is to discretize problem
(\ref{Pr_cauchy_simple}) with the help of a Finite Element Method.
More precisely, we assume that $\Omega$ supports a triangular mesh which is regular in the sense of \cite{ciarlet}, the maximal diameter of each triangle being $h$. 
Let us denote by $V_{0,h}$ and $\tilde{V}_{0,h}$ the finite dimensional subspaces of $V_0$ and $\tilde{V}_0$, respectively, formed by the continuous functions on $\overline{\Omega}$ which are affine on each triangle and which vanish on the sides which belong to $\overline{\Gamma}$  and $\overline{\tilde{\Gamma}}$, respectively.
The discretized version of the mixed formulation of quasi-reversibility (\ref{Pr_cauchy_simple}) is: for 
$\eps >0$, find $(u_{\eps,h},\lambda_{\eps,h}) \in V_{0,h} \times \tilde{V}_{0,h}$ such that for all $(v_h,\mu_h) \in V_{0,h} \times \tilde{V}_{0,h}$,
\begin{equation} 
\left\{\begin{array}{rcl} 
\displaystyle \eps \int_\Omega \nabla u_{\eps,h} \cdot \nabla v_h\,dx + \int_\Omega \nabla v_h\cdot \nabla \lambda_{\eps,h}\,dx&=&0\\[8pt]
\displaystyle \int_\Omega \nabla u_{\eps,h} \cdot \nabla \mu_h\,dx - \int_\Omega \nabla \lambda_{\eps,h}\cdot \nabla \mu_h\,dx&=&\dsp\int_\Omega f\mu_h\,dx.
\end{array} \right. 
\label{Pr_cauchy_simple_h} 
\end{equation}
We denote $(u^\delta_{\eps,h},\lambda^\delta_{\eps,h})$ the solution to problem (\ref{Pr_cauchy_simple_h}) which is associated with the noisy data $f^\delta$ instead of the exact data $f$.
In practice, the solution $u^\delta_{\eps,h}$ is the only approximate function of the exact solution $u$ which is accessible, this is why we are interested in the norm of the discrepancy $u^\delta_{\eps,h}-u$ in the domain $\Omega$.
In this view, we write
\be \|u^\delta_{\eps,h}-u\|_{H^1(\Omega)} \leq \|u^\delta_{\eps,h} -u_{\eps,h}\|_{H^1(\Omega)}+ \|u_{\eps,h}-u_\eps\|_{H^1(\Omega)}+ \|u_\eps-u\|_{H^1(\Omega)}\label{triangular},\ee
and estimate each term of this decomposition. The first term to estimate corresponds to the error due to the noisy data.
Let us prove the following lemma.
\begin{lemma}
There exists a constant $C>0$ which depends only on the geometry such that
\be \|u^\delta_{\eps,h} -u_{\eps,h}\|_{H^1(\Omega)} \leq C\frac{\delta}{\sqrt{\eps}}.\label{first_estimate}\ee
\end{lemma}
\begin{proof}
By reusing the bilinear form $A_\eps$ introduced in the proof of Theorem \ref{base}, $(u^\delta_{\eps,h},\la^\delta_{\eps,h})$ and $(u_{\eps,h},\la_{\eps,h})$ are solutions in $V_{0,h} \times \tilde{V}_{0,h}$ to the weak problems: for all $(v_h,\mu_h) \in V_{0,h} \times \tilde{V}_{0,h}$,
\[A_\eps((u^\delta_{\eps,h},\la^\delta_{\eps,h});(v_h,\mu_h))=-\int_\Omega f^\delta \mu_h\,dx,\qquad A_\eps((u_{\eps,h},\la_{\eps,h});(v_h,\mu_h))=-\int_\Omega f \mu_h\,dx.\]
Taking the difference, setting $(v_h,\mu_h)=(u^\delta_{\eps,h}-u_{\eps,h},\la^\delta_{\eps,h}-\la_{\eps,h})$, we get
\[\eps \|u^\delta_{\eps,h}-u_{\eps,h}\|^2+ \|\la^\delta_{\eps,h}-\la_{\eps,h}\|^2\leq \|f^\delta-f\|_{L^2(\Omega)} \|\la^\delta_{\eps,h}-\la_{\eps,h}\|_{L^2(\Omega)},\]
which completes the proof by using the Poincar\'e inequality and (\ref{noise}).  
\end{proof}
\noindent The second term of (\ref{triangular}) corresponds to the error due to discretization.
Let us prove the following lemma, which is a consequence of Theorem \ref{main}.
\begin{lemma}\label{Second}
There is a constant $C>0$ which depends only on the geometry and on $u$ such that
\be \|u_{\eps,h} -u_{\eps}\|_{H^1(\Omega)} \leq C\,\frac{h^{s-1}}{\eps},\label{second_estimate}\ee
where $s$ is given in the statement of Theorem \ref{main}.
\end{lemma}
\begin{proof}
The proof relies in particular on C\'ea's Lemma. Since we need a uniform estimate with respect to $\eps$, we detail the proof.
For all $(v_h,\mu_h) \in V_{0,h} \times \tilde{V}_{0,h}$, we have
\[A_\eps((u_{\eps}-u_{\eps,h},\la_\eps-\la_{\eps,h});(v_h,\mu_h))=0.\]
This implies that for all $(v_h,\mu_h) \in V_{0,h} \times \tilde{V}_{0,h}$, 
\[A_\eps((u_{\eps}-u_{\eps,h},\la_\eps-\la_{\eps,h});(u_{\eps}-u_{\eps,h},\la_\eps-\la_{\eps,h}))=A_\eps((u_{\eps}-u_{\eps,h},\la_\eps-\la_{\eps,h});(u_{\eps}-v_h,\la_\eps-\mu_h)),\]
hence
\[
\begin{array}{ll}
&\dsp  A_\eps((u_{\eps}-u_{\eps,h},\la_\eps-\la_{\eps,h});(u_{\eps}-u_{\eps,h},\la_\eps-\la_{\eps,h}))\\[9pt]
\leq & \dsp \inf_{(v_h,\mu_h) \in V_{0,h} \times \tilde{V}_{0,h}} \left|A_\eps((u_{\eps}-u_{\eps,h},\la_\eps-\la_{\eps,h});(u_{\eps}-v_h,\la_\eps-\mu_h))\right|.
\end{array}
\]
But on the one hand, we have
\[A_\eps((u_{\eps}-u_{\eps,h},\la_\eps-\la_{\eps,h});(u_{\eps}-u_{\eps,h},\la_\eps-\la_{\eps,h}))=\eps\|u_{\eps}-u_{\eps,h}\|^2+\|\la_\eps-\la_{\eps,h}\|^2
\]
while on the other hand, there holds
\[
\begin{array}{ll}
&\dsp\inf_{(v_h,\mu_h) \in V_{0,h} \times \tilde{V}_{0,h}}\left|A_\eps((u_{\eps}-u_{\eps,h},\la_\eps-\la_{\eps,h});(u_{\eps}-v_h,\la_\eps-\mu_h))\right| \\[9pt]
\leq & \dsp(\eps \|u_\eps-u_{\eps,h}\| + \|\la_\eps-\la_{\eps,h}\|) \inf_{v_h \in V_{0,h}}\|u_\eps-v_h\| + (\|u_\eps-u_{\eps,h}\| + \|\la_\eps-\la_{\eps,h}\|) \inf_{\mu_h \in \tilde{V}_{0,h}}\|\la_\eps-\mu_h\|.
\end{array}
\]
By using the classical interpolation error estimates in $H^s(\Omega)$ for $s>1$ (see \cite{grisvard_bleu}), we know that there exists a constant $C>0$ which depends only on the geometry such that
\[ \inf_{v_h \in V_{0,h}}\|u_\eps-v_h\| \leq C\,h^{s-1}\|u_\eps\|_{H^s(\Omega)},\qquad\qquad \inf_{\mu_h \in \tilde{V}_{0,h}}\|\la_\eps-\mu_h\| \leq C\,h^{s-1}\|\la_\eps\|_{H^s(\Omega)}.\]
Theorem \ref{main} in the case of exact data $f$ implies that there is a constant $C>0$ which depends on the geometry and on $u$ such that
\[\|u_\eps\|_{H^{s}(\Omega)} \leq C\,\frac{1}{\sqrt{\eps}}\qquad\mbox{ and }\qquad  \|\la_\eps\|_{H^{s}(\Omega)} \leq C.\]
From the three above estimates, we get
\[
\begin{array}{ll}
&\dsp \inf_{(v_h,\mu_h) \in V_{0,h} \times \tilde{V}_{0,h}}\left|A_\eps((u_{\eps}-u_{\eps,h},\la_\eps-\la_{\eps,h});(u_{\eps}-v_h,\la_\eps-\mu_h))\right| \\[9pt]
\leq & \dsp C\,\frac{h^{s-1}}{\sqrt{\eps}}(\eps \|u_\eps-u_{\eps,h}\| + \|\la_\eps-\la_{\eps,h}\|)+ C\,h^{s-1}(\|u_\eps-u_{\eps,h}\| + \|\la_\eps-\la_{\eps,h}\|) \\[9pt]
\leq &\dsp C\, \frac{h^{s-1}}{\sqrt{\eps}}(\sqrt{\eps}\|u_\eps-u_{\eps,h}\| + \|\la_\eps-\la_{\eps,h}\|).
\end{array}
\]
Eventually we end up with
\[\eps\|u_{\eps}-u_{\eps,h}\|^2+\|\la_\eps-\la_{\eps,h}\|^2 \leq  C\, \frac{h^{s-1}}{\sqrt{\eps}}(\eps\|u_\eps-u_{\eps,h}\|^2 + \|\la_\eps-\la_{\eps,h}\|^2)^{1/2},\]
which completes the proof.
\end{proof}
\noindent Estimating the third term in (\ref{triangular}) is strongly related to the stability of the Cauchy problem for the Laplace equation, a topic which has a long history since the pioneering paper \cite{Hada02} (see e.g. \cite{Payn60,Payn70,AxBE06,benbelgacem,alessandrini_rondi_rosset_vessella,phung,bourgeois2,BoDa10}).
It is well-known that since such problem is exponentially ill-posed, the corresponding stability estimate is at best of logarithmic type (see for example \cite{bourgeois2}). To our best knowledge, an estimate of $\eta(\eps):=\|u_\eps-u\|_{H^1(\Omega)}$, which tends to $0$ when $\eps$ tends to $0$ in view of Theorem \ref{base_simple}, is unknown.
However, a logarithmic stability estimate for $\|u_\eps-u\|_{L^2(\Omega)}$ can be derived from Theorem 1.9 in \cite{alessandrini_rondi_rosset_vessella}.
\begin{lemma}
\label{logarithmic}
There exists a constant $C>0$ which depends only on the geometry and on $u$ and a constant $\mu \in (0,1)$ which depends only on the geometry such that
\[\|u_\eps-u\|_{L^2(\Omega)} \leq C \frac{1}{\big(\log(1/\eps)\big)^\mu}.\]
\end{lemma} 
\begin{proof}
From (\ref{cauchy_strong_simple}) and (\ref{Pr_cauchy_simple_strong}), the functions $u_\eps-u$ and $\la_\eps$ satisfy
\[\left\{
\begin{array}{rcll}
 -\Delta (u_\eps-u)  &=& -\eps f/(1+\eps)&\mbox{in } \Omega\\
 u_\eps-u&=&0  & \mbox{on }\Gamma \\
 \partial_\nu (u_\eps-u)&=& \partial_\nu \la_\eps & \mbox{on }\Gamma.
\end{array}
\right.
\]
By using the estimate (\ref{estim2}) of Theorem \ref{base_simple}, we get
\[\|u_\eps-u\|_{H^1(\Omega)} \leq C,\qquad \|\Delta (u_\eps-u)\|_{L^2(\Omega)} \leq C\,\eps,\qquad \|\partial_\nu (u_\eps-u)\|_{H^{-1/2}(\Gamma)} \leq C\, \sqrt{\eps}.\]
By plugging these estimates in Theorem 1.9 of \cite{alessandrini_rondi_rosset_vessella}, we obtain the result.
\end{proof}
\noindent In conclusion, by gathering (\ref{triangular}), (\ref{first_estimate}) and (\ref{second_estimate}),
we end up with the final estimate
\be \|u^\delta_{\eps,h}-u\|_{H^1(\Omega)} \leq C\,\frac{\delta}{\sqrt{\eps}}+C\,\frac{h^{s-1}}{\eps}+ \eta(\eps),\label{estim_final}\ee
where $s$ is given in the statement of Theorem \ref{main} and 
$\eta$ converges to $0$ when $\eps$ tends to $0$ at best with a logarithmic convergence rate in view of Lemma \ref{logarithmic}.
An important application of the estimate (\ref{estim_final}) is that when $\delta \rightarrow 0$, we have to choose $\eps=\eps(\delta)$ and $h=h(\eps)$ such that
\[\lim_{\delta \rightarrow 0}\frac{\delta}{\sqrt{\eps(\delta)}}=0,\qquad\qquad \lim_{\eps \rightarrow 0} \frac{h^{s-1}(\eps)}{\eps}=0\]
in order to obtain a good approximation of the exact solution from noisy data and by using our Finite Element Method.
\section*{Appendix A: A basic uniform estimate}
For $\la \in \mathbb{C}$, we introduce the symbol $\mathscr{J}(\la) : \mathcal{D}
(\mathscr{J}) \longrightarrow L^2(0,\omega)$ where 
\[
\mathcal{D}(\mathscr{J})=\{ u \in 
H_{0,0}^1(0,\omega) \cap H^{2}(0,\omega),\,\, d_\theta u(\omega)=0\}\]
and
\[\mathscr{J}(\la)\varphi=-(\la^2+d^2_\theta)\varphi.\]
The goal of the appendix is to establish the following result.
\begin{proposition}
\label{Estimate_appendix}
If ${\rm Re}\,\la \notin\{(\pi/2+n\pi)/\om,\,n\in\mathbb{Z}\}$, then $\mathscr{J}$ is an isomorphism and if $\varphi \in \mathcal{D}(\mathscr{J})$ satisfies
$\mathscr{J}(\varphi)=g \in L^2(0,\omega)$,
we have 
the estimate
\be \|d^2_\theta\varphi\|_{L^2(0,\omega)} +|\la|^2 \|\varphi\|_{L^2(0,\omega)} \leq C\,\|g\|_{L^2(0,\omega)},\label{estimate}\ee
where $C>0$ is independent of $g$ and ${\rm Im}\,\la$.
\end{proposition}
\noindent To prove proposition \ref{Estimate_appendix}, we need three lemmas. 
We first consider a simple situation when $\la$ is purely imaginary.
\begin{lemma}
\label{itau}
If $\la=i\tau$, $\tau \in \mathbb{R}$, the mapping $\mathscr{J}$ is an isomorphism and if $\varphi \in \mathcal{D}(\mathscr{J})$ satisfies $\mathscr{J}(\varphi)=g \in L^2(0,\omega)$, we have
\[ \|d^2_\theta\varphi\|_{L^2(0,\omega)} +|\la|^2 \|\varphi\|_{L^2(0,\omega)} \leq 3\,\|g\|_{L^2(0,\omega)}.\]
\end{lemma}
\begin{proof}
For $\la=i\tau$ with $\tau \in \mathbb{R}$, due to the Lax-Milgram lemma and Poincar\'e inequality, for all $g \in L^2(0,\omega)$ there exists a unique $\varphi \in H_{0,\omega}^1(0,\om)$ such that $(\tau^2-d^2_\theta)\varphi=g$ and $d_\theta \varphi(\omega)=0$.
Then $d^2_\theta\varphi=\tau^2 \varphi-g \in L^2(0,\omega)$. Hence $\mathscr{J}(\la)$ is invertible and continuous. From the Banach theorem, $\mathscr{J}(\la)$ is an isomorphism.
More precisely, the Lax-Milgram lemma implies that
\[\|d_\theta\varphi\|^2_{L^2(0,\omega)}+|\la|^2\|\varphi\|^2_{L^2(0,\omega)}=(g,\varphi)_{L^2(0,\omega)},\]
in particular
\[|\la|^2\|\varphi\|_{L^2(0,\omega)} \leq \|g\|_{L^2(0,\omega)}.\]
Since in addition $d^2_\theta\varphi=\tau^2 \varphi-g$, we have
\[ \|d^2_\theta\varphi\|_{L^2(0,\omega)} \leq |\la|^2 \|\varphi\|_{L^2(0,\omega)} + \|g\|_{L^2(0,\omega)} \leq 2\,\|g\|_{L^2(0,\omega)},\]
which completes the proof.
\end{proof}
\noindent We will say that $\la \in \mathbb{C}$ is an eigenvalue of $\mathscr{J}$ if ${\rm Ker}\,\mathscr{J}(\la) \neq \{0\}$.
We have the following lemma.
\begin{lemma}
\label{eigenvalue}
For all $\la \in \mathbb{C}$, $\mathscr{J}(\la) : \mathcal{D}(\mathscr{J}) \longrightarrow L^2(0,\omega)$ is an isomorphism
if and only if $\la$ is not one of the $\la_n=(\pi/2+n \pi)/\omega$, $n \in \mathbb{Z}$.
\end{lemma}
\begin{proof}
Lemma \ref{itau} indicates that the result is true for any $\la \in i\mathbb{R}$. It follows from the analytic Fredholm theorem that
$\mathscr{J}(\la) : \mathcal{D}(\mathscr{J}) \longrightarrow L^2(0,\omega)$ is an isomorphism
if and only if $\la$ is not an eigenvalue of $\mathscr{J}$.
It is  
straightforward that 
the eigenvalues of $\mathscr{J}$ are $\la_n=(\pi/2+n \pi)/\omega$, $n \in \mathbb{Z}$,
the corresponding eigenfunctions being given by
$\varphi_n(\theta)=\sin((\pi/2+n\pi)\theta/\omega)$. The result follows.
\end{proof}
\noindent We now consider a situation where $\la$ is no longer purely imaginary.
\begin{lemma}\label{papillon}
There exists a real positive constant $\delta$ such that for all $\la \in \mathbb{C}$ satisfying 
\[|{\rm Re}\,\la| < \delta\, |{\rm Im}\,\la|,\]
the operator $\mathscr{J}$ is an isomorphism and if $\varphi \in \mathcal{D}(\mathscr{J})$
satisfies $\mathscr{J}(\la)\varphi=g \in L^2(0,\omega)$, then 
\[ \|d^2_\theta\varphi\|_{L^2(0,\omega)} +|\la|^2 \|\varphi\|_{L^2(0,\omega)} \leq 4\,\|g\|_{L^2(0,\omega)}.\]
\end{lemma}
\begin{proof}
We already know from Lemma \ref{itau} that the result holds for $\la \in i \mathbb{R}$.
Now let us consider the case when $\la \notin i \mathbb{R}$.
We write $\la$ as $\la=\pm i |\la| e^{i\psi}$ for $\psi \in (-\pi/2,\pi/2)$.
Set $\tilde{\la}=\pm i |\la|$. Since $|\la|=|\tilde{\la}|$, 
we have 
\[\|d^2_\theta\varphi\|_{L^2(0,\omega)}+|\la|^2\|\varphi\|_{L^2(0,\omega)}=\|d^2_\theta\varphi\|_{L^2(0,\omega)}+|\tilde{\la}|^2\|\varphi\|_{L^2(0,\omega)}.\]
Let us define $\tilde{g}=\mathscr{J}(\tilde{\la})\varphi$.
According to Lemma \ref{itau}, we have
\[ \|d^2_\theta\varphi\|_{L^2(0,\omega)}+|\la|^2\|\varphi\|_{L^2(0,\omega)} \leq 3\,\|\tilde{g}\|_{L^2(0,\omega)}.\]
We have that
\[\|\tilde{g}\|_{L^2(0,\omega)} \leq \|g\|_{L^2(0,\omega)} + \|\tilde{g}-g\|_{L^2(0,\omega)}\]
and
\[\|\tilde{g}-g\|_{L^2(0,\omega)} =\|\mathscr{J}(\tilde{\la})\varphi-\mathscr{J}(\la)\varphi\|_{L^2(0,\omega)} \leq |\tilde{\la}^2-\la^2| \|\varphi\|_{L^2(0,\omega)}.\]
We obtain that
\[\|\tilde{g}-g\|_{L^2(0,\omega)}  \leq |e^{2i \psi}-1|^2 |\la|^2 \|\varphi\|_{L^2(0,\omega)}.\]
For all $\eps>0$, there exist $\delta$ small enough such that $\|\tilde{g}-g\|_{L^2(0,\omega)}  \leq \eps|\la|^2\|\varphi\|_{L^2(0,\omega)}$. By choosing $3\eps=1/4$ we eventually obtain the result. 
\end{proof}
\begin{proof}[Proof of Proposition \ref{Estimate_appendix}]
Lemma \ref{papillon} implies that for all $\la \in \mathbb{C}$ such that ${\rm Re}\,\la=\beta$ and $|{\rm Im}(\la)| \geq \nu_\beta$, we have the estimate
\[\|d^2_\theta\varphi\|_{L^2(0,\omega)} +|\la|^2 \|\varphi\|_{L^2(0,\omega)} \leq C\,\|g\|_{L^2(0,\omega)},
\]
where $C>0$ is independent of $\la$, $g$ and $\nu_\beta$ depends only on $\beta$.
For $\la \in [\beta-i\nu_\beta,\beta+i \nu_\beta]$, the symbol $\mathscr{J}(\la)$ is invertible according to 
Lemma \ref{eigenvalue}. The analytic Fredholm theorem guarantees that the inverse operator $\la \mapsto \mathscr{J}(\la)^{-1}$ is continuous outside of its poles. Since the segment $[-\beta-i\nu_\beta,-\beta+i \nu_\beta]$ is compact, we deduce that the above estimate remains true for all $\la$ such that ${\rm Re}\,\la=\beta$ with a constant $C$ which depends neither on $g$ nor ${\rm Im}\,\la$.
\end{proof}

\section*{Appendix B: Proofs of Lemmas \ref{lemmaEstimate} and \ref{lemmaEstimate2}}
In order to prove Lemmas \ref{lemmaEstimate} and \ref{lemmaEstimate2},
we will need the following
formulas, which hold for any $\la \in \mathbb{C}$ and $\theta \in \mathbb{R}$,
\[
\cos(\la \theta)=\cos({\rm Re}(\la)\theta)\cosh ({\rm Im}(\la)\theta)-i \sin({\rm Re}(\la)\theta)\sinh({\rm Im}(\la)\theta)
\]
and
\[
\sin(\la \theta)=\sin({\rm Re}(\la)\theta)\cosh ({\rm Im}(\la)\theta)+i \cos({\rm Re}(\la)\theta)\sinh({\rm Im}(\la)\theta).
\]
They imply
\begin{equation}\label{RelationSin}
|\sin(\lambda\theta)|^2=(\cosh(2{\rm Im}(\lambda)\theta)-\cos(2{\rm Re}(\lambda)\theta))/2
\end{equation}
and 
\begin{equation}\label{RelationCos}
|\cos(\lambda\theta)|^2=(\cosh(2{\rm Im}(\lambda)\theta)+\cos(2{\rm Re}(\lambda)\theta))/2.
\end{equation}
In the following lemmas, we give the proof of two technical results needed in the previous analysis.
\begin{lemma}\label{lemmaEstimate}
Assume that $\beta\notin\{(\pi/2+n\pi)/\om,\,n\in\mathbb{Z}\}$. There is a constant $C>0$ independent of $\eps>0$, $\lambda=\beta+i\tau\in\ell_{\beta}$ such that
\begin{equation}\label{EstimationLaborieuse}
\cfrac{e^{2|\tau|\omega}}{|1+\eps\,\cos^2(\lambda\omega)|^2} \le C/\eps.
\end{equation}
\end{lemma}
\begin{proof}
Observing that $e^{2|\tau|\omega}\le 4\,\cosh(\tau\omega)^2$, we see that to establish (\ref{EstimationLaborieuse}), it is sufficient to show that there is some $\eta>0$ such that
\begin{equation}\label{EstimationLaborieuse1}
\cfrac{\eta\sqrt{\eps}\cosh(\tau\om)}{|1+i\sqrt{\eps}\,\cos(\lambda\om)|^2}\,\cfrac{\eta\sqrt{\eps}\cosh(\tau\om)}{|1-i\sqrt{\eps}\,\cos(\lambda\om)|^2} \le 1.
\end{equation}
We will study the two factors on the left hand side of  (\ref{EstimationLaborieuse1}) proving that for $\eta>0$ small enough they are both smaller than one. Let us consider the first one. A direct computation  gives 
\begin{equation}\label{CalculusReIm}
|1+i\sqrt{\eps}\,\cos(\lambda\om)|^2=\eps\cos(\beta\om)^2\cosh(\tau\om)^2+(1+\sqrt{\eps}\sin(\beta\om)\sinh(\tau\om))^2.
\end{equation}
Define the polynomial function $P$ such that
\[
P(X)=X^2\cos(\beta\om)^2\cosh(\tau\om)^2+(1+X\sin(\beta\om)\sinh(\tau\om))^2-\eta X\cosh(\tau\om).
\]
We see that the first factor on the left hand side of  (\ref{EstimationLaborieuse1}) is smaller than one as soon as $P$ is positive on $\mathbb{R}$. Since $P(0)=1>0$, it is sufficient to show that its discriminant is negative. We find 
\[
\begin{array}{lcl}
\Delta_P&=&(2\sin(\beta\om)\sinh(\tau\om)+\eta\cosh(\tau\om))^2-4(\cos(\beta\om)^2\cosh(\tau\om)^2+\sin(\beta\om)^2\sin(\tau\om)^2)\\
&=&\Big((\eta^2-4\cos(\beta\om)^2)\cosh(\tau\om)+4\eta\sin(\beta\om)\sinh(\tau\om)\Big)\cosh(\tau\om). 
\end{array}
\]
Observing that $|\sinh(\tau\om)|<\cosh(\tau\om)$, we can write
\[
\begin{array}{ll}
&(\eta^2-4\cos(\beta\om)^2)\cosh(\tau\om)+4\eta\sin(\beta\om)\sinh(\tau\om)\\
\le &(\eta^2+4\eta|\sin(\beta\om)|-4\cos(\beta\om)^2)\cosh(\tau\om).
\end{array}
\]
Therefore, since $\cos(\beta\om)\neq 0$ when $\beta\notin\{(\pi/2+n\pi)/\om,\,n\in\mathbb{Z}\}$, we see that we can find $\eta>0$ small enough (but independent of $\tau$) such that $\Delta_P<0$. This shows that the first factor on the left hand side of (\ref{EstimationLaborieuse1}) is smaller than one. A completely similar approach allows one to prove that the second factor is also smaller than one. As a consequence, (\ref{EstimationLaborieuse1}) is satisfied for $\eta$ small enough and so is (\ref{EstimationLaborieuse}).
\end{proof}
\begin{lemma}\label{lemmaEstimate2}
Assume that $\beta\notin\{(\pi/2+n\pi)/\om,\,n\in\mathbb{Z}\}$. There is a constant $C>0$ independent of $\eps>0$, $\lambda=\beta+i\tau\in\ell_{\beta}$ such that
\begin{equation}\label{EstimationLaborieuse_bis}
\cfrac{\eps^2e^{4|\tau|\om}}{|1+\eps\,\cos^2(\lambda\om)|^2} \le C.
\end{equation}
\end{lemma}
\begin{proof}
As in the proof of Lemma \ref{lemmaEstimate}, one can check that it is sufficient to show that there is some $\eta>0$ such that
\begin{equation}\label{EstimationLaborieuse1_bis}
\cfrac{\eta\eps\cosh(\tau\om)^2}{|1+i\sqrt{\eps}\,\cos(\lambda\om)|^2}\,\cfrac{\eta\eps\cosh(\tau\om)^2}{|1-i\sqrt{\eps}\,\cos(\lambda\om)|^2} \leq 1.
\end{equation}
In (\ref{CalculusReIm}), we obtained 
\begin{equation}\label{CalculusReImBis}
|1\pm i\sqrt{\eps}\,\cos(\lambda\om)|^2=\eps\cos(\beta\om)^2\cosh(\tau\om)^2+(1\mp\sqrt{\eps}\sin(\beta\om)\sinh(\tau\om))^2.
\end{equation}
Therefore, we can write
\[
\begin{array}{ll}
& |1\pm i\sqrt{\eps}\,\cos(\lambda\om)|^2-\eta\eps\cosh(\tau\om)^2 \\
= & \eps(\cos(\beta\om)^2-\eta)\cosh(\tau\om)^2+(1\mp\sqrt{\eps}\sin(\beta\om)\sinh(\tau\om))^2 >0
\end{array}
\]
for $\eta$ small enough. This is enough to conclude.
\end{proof}
\bibliography{Biblio}
\bibliographystyle{plain}

 \end{document}